\documentclass[10pt,a4paper,american]{amsart}
\usepackage[left=3.5cm, right=3.5cm]{geometry}
\usepackage[american]{babel}
\usepackage{amsthm}
\usepackage{amsbsy}
\usepackage{amstext}
\usepackage{amsfonts}
\usepackage{mathcomp}
\usepackage{dsfont}
\usepackage{latexsym}
\usepackage{amssymb}
\usepackage{esint}
\usepackage[colorlinks,pdfpagelabels,pdfstartview = FitH,bookmarksopen = true,bookmarksnumbered = true,linkcolor = blue,plainpages = false,hypertexnames = false,citecolor = red,pagebackref=true]{hyperref}
\usepackage{cite}
\usepackage{stmaryrd}
\usepackage{comment}

\newtheorem{theo}{Theorem}[section]
\newtheorem{cor}[theo]{Corollary}

\newtheorem{lem}[theo]{Lemma}

\theoremstyle{definition}
\newtheorem{defin}[theo]{Definition}
\newtheorem*{lem*}{Lemma}
\newtheorem{rem}[theo]{Remark}

\newtheorem*{cor*}{Corollary}
\newtheorem*{theo*}{Theorem}

\newcommand{\norm}[1]{\lVert#1\rVert}

\def\Ex#1{[\![#1]\!]_h}
\def\Exx#1{[\![#1]\!]_{\bar h}}

\DeclareMathOperator*{\esssup}{ess\,sup}

\def\XXint#1#2#3{{\setbox0=\hbox{$#1{#2#3}{\int}$ }
\vcenter{\hbox{$#2#3$ }}\kern-.6\wd0}}

\def\XXiint#1#2#3{{\setbox0=\hbox{$#1{#2#3}{\iint}$ }
\vcenter{\hbox{$#2#3$ }}\kern-.55\wd0}}

\renewcommand{\d}{\:\:\!\!\mathrm{d}}
\newcommand{\N}{\ensuremath{\mathbb{N}}}
\newcommand{\R}{\ensuremath{\mathbb{R}}}

\newcommand{\A}{\mathcal{A}}
\newcommand{\wto}{\rightharpoondown}

\renewcommand{\b}{\mathfrak{b}}

\renewcommand{\epsilon}{\varepsilon}
\renewcommand{\rho}{\varrho}

\numberwithin{equation}{section}
\allowdisplaybreaks

\begin{document}
\renewcommand{\refname}{References} 
\renewcommand{\abstractname}{Abstract}

\title[Existence of weak solutions to a DSM equation]{Existence of solutions to a diffusive\\ shallow medium equation}

\date{\today}
 
\makeatother

\author[V. B\"ogelein]{Verena B\"{o}gelein}
\address{Verena B\"ogelein\\
Fachbereich Mathematik, Universit\"at Salzburg\\
Hellbrunner Str. 34, 5020 Salzburg, Austria}
\email{verena.boegelein@sbg.ac.at}

\author[N. Dietrich]{Nicolas Dietrich}
\address{Nicolas Dietrich\\
Fachbereich Mathematik, Universit\"at Salzburg\\
Hellbrunner Str. 34, 5020 Salzburg, Austria}
\email{nicolas.dietrich@stud.sbg.ac.at}

\author[M. Vestberg]{Matias Vestberg}
\address{Matias Vestberg\\
Department of Mathematics, Aalto University\\
P.~O.~Box 11100, FI-00076 Aalto University, Finland}
\email{matias.vestberg@aalto.fi}

\begin{abstract}
In this article we establish the existence of weak solutions to the shallow medium equation. We proceed by an approximation argument. First we truncate the coefficients of the equation from above and below. Then we prove convergence of the solutions of the truncated problem to a solution to the original equation. 
\end{abstract}

\maketitle
\tableofcontents

\section{Introduction}

In this work we prove the existence of weak solutions to the Cauchy-Dirichlet problem associated with the doubly nonlinear equation
\begin{align}\label{DNPE}
\partial_t u - \nabla\cdot \big((u-z)^\alpha |\nabla u|^{p-2}\nabla u\big) = f  \quad \text{ in } \Omega_T:=\Omega\times (0,T),
\end{align}
where $\Omega\subset \R^n$ is an open bounded set, $z:\Omega\to \R$ and $f:\Omega_T\to \R$ are given sufficiently regular functions, and the parameters $\alpha$ and $p$ are restricted to the range
$$
p>1, \qquad \text{ and }\qquad  \alpha > 0. 
$$
The solution $u$ is required to satisfy $u\geq z$ so that the factor $(u-z)^\alpha$ is well-defined. The term doubly nonlinear refers to the fact that the diffusion part depends nonlinearly both on the gradient and the solution itself. In this article we consider all $n\geq 2$, although the equation is typically used in models with only two spatial dimensions.

In the case $z=0$, \eqref{DNPE} reduces to a standard doubly nonlinear equation which was studied already in \cite{Ve} and \cite{PoVe}, although using a different notion of weak solution. If we furthermore set $\alpha=0$, we recover the parabolic $p$-Laplace equation, whereas taking $p=2$ we obtain the porous medium equation, after a formal application of the chain rule. Thus, \eqref{DNPE} can be seen as a generalization of many previously studied equations.

Local boundedness of weak solutions to \eqref{DNPE} has been proved in \cite{SiVe} for sufficiently regular $f$ and $z$ in the slow diffusion case $\alpha +p >2$. The article focused on the case $p<2$, but the value of $p$ does not really play a role in the arguments. In \cite{SiVe2} local H\"older continuity was established for $p>2$. Existence of solutions to a Cauchy-Dirichlet problem corresponding to \eqref{DNPE} was proved in the case of trivial topography $z=0$ and slow diffusion $\alpha + p >2$ in \cite{AlSaDa}. See also \cite{AltLu, BoDuMaSc, IMJ, St} which study the existence of solutions to the Cauchy-Dirichlet problem for doubly nonlinear equations with general structure conditions.

In the range $p > 2$, equation \eqref{DNPE} is used to describe the dynamics of glaciers in the so-called shallow ice approximation, see for example Chapter 2 of \cite{Pa} and the classic work \cite{Ma}. Laboratory experiments, theoretical considerations and field measurments suggest that in these applications $p\approx 4$ provides the best description. This corresponds to the value $3$ for the exponent appearing in the flow law for polycrystalline ice which was established in \cite{Ho}. In the range $p<2$, equation \eqref{DNPE} is used to model shallow water dynamics in situations such as floods and dam breaks, see \cite{AlSaDa,FeMo,HrBeFr}. Due to the variety of applications we follow the terminology from \cite{SiVe2} and use the term diffusive shallow medium equation or more concisely, DSM equation, to describe equation \eqref{DNPE} for all $p>1$.

The given function $z$ describes the elevation of the land on top of which the water or ice is moving, measured with respect to some arbitrarily chosen ground level. The value of $u$ represents the height of the medium measured with respect to the same level. This provides a physical motivation for the condition $u\geq z$ which was imposed earlier. Even though $u$ and $z$ depend on the ground level, their difference $v:=u-z$ is invariant. Therefore, it is natural to reformulate the equation in terms of $v$, see \eqref{reformulated} and Definition~\ref{weakdef} of Section \ref{weaksolsect} below. As has been established in \cite{SiVe,SiVe2}, this formulation is also mathematically convenient. 

The right-hand side $f$ is a source term which can account for snowfall in the case of glaciers, and rainfall, evaporation or infiltration in the shallow water setting. Thus, its sign may vary in applications. However, in this paper we restrict ourselves to nonnegative source terms, which is necessary for the comparison principle argument utilized in Lemma \ref{lem:lower-bound-vk}. In physical applications it is reasonable to assume that $f$ is bounded, but our existence result requires only that the source term has sufficiently high integrability, see Section \ref{weaksolsect} for details.

The function $z$ is required to belong to a first order Sobolev space with sufficiently high exponent, see \eqref{assumption:data}. Infact, the exponent is higher than the dimension of $\Omega$, which means that $z$ is locally H\"older continuous. The space of admissible $z$ contains for example all Lipschitz functions, which include the topographies occurring in realistic diffusion models.

Although the shallow medium equation has been studied in the literature, to our knowledge existence of weak solutions has not been established except for the trivial topography $z=0$. Our aim in this article is to close this gap. Our main result, Theorem~\ref{thm:existence}, ensures the existence of weak solutions to the Cauchy-Dirichlet problem of the shallow medium equation \eqref{DNPE}. 

\medskip
 
\noindent
{\bf Acknowledgments.} V.~B\"ogelein has been supported by the FWF-Project P31956-N32
``Doubly nonlinear evolution equations". M.~Vestberg wants to express gratitude to the Academy of Finland.

\section{Setting and main result}\label{weaksolsect}

In the following we explain the precise setting and present the natural definition of weak solutions which was introduced in \cite{SiVe}. 
For the right-hand side $f \colon \Omega_T \to \R_{\ge 0}$ and the known function $z\colon\Omega\to \R$ and the initial datum $\psi\colon\Omega\to\R$ we assume
\begin{align}
\label{assumption:data}
 f\in L^{\sigma p'}(\Omega_T;\R_{\geq 0}), \hspace{7mm} 
 z\in W^{1,\sigma\beta p}(\Omega), \hspace{7mm} 
 \Psi :=\psi-z \in L^\infty(\Omega;\R_{\geq 0}), 
\end{align}
for some $\sigma>\frac{n+p}{p}$ and where $p'=\frac{p}{p-1}$ denotes the H\"older conjugate of $p$. We consider the Cauchy-Dirichlet problem 
\begin{align}\label{CD}
\left\{
\begin{array}{ll}
\partial_t u - \nabla\cdot \big((u-z)^\alpha |\nabla u|^{p-2}\nabla u\big) = f  &\quad \text{ in } \Omega_T,\\[5pt]
u=z & \quad \text{ on } \partial\Omega\times(0,T),\\[5pt]
u(\cdot,0)=\psi & \quad \text{ in } \overline\Omega,
\end{array}
\right.
\end{align}
for $u\colon\Omega_T\to\R$ with $u\ge z$. 
In order to motivate the natural definition of weak solutions, we reformulate \eqref{CD}$_1$ in terms of $v:=u-z$. Formally applying the chain rule as in \cite{SiVe}, we can write equation \eqref{CD}$_1$ in the form
\begin{align}\label{reformulated}
\partial_t v-\nabla \cdot \big(\beta^{1-p}|\nabla v^\beta+\beta v^{\beta-1}\nabla z|^{p-2}(\nabla v^\beta +\beta v^{\beta-1}\nabla z)\big)= f,
\end{align}
where 
\begin{align}\label{def:beta}
\beta:=1 + \frac{\alpha}{p-1}>1.
\end{align}
Throughout the article, we will focus on this form of the equation. In order to simplify the notation, we denote the vector field appearing in the diffusion part of \eqref{reformulated} as
\begin{align}\label{def:A}
A(v,\xi):= \beta^{1-p}|\xi+\beta v^{\beta-1}\nabla z|^{p-2}(\xi +\beta v^{\beta-1}\nabla z),
\end{align}
for $v\ge 0$ and $\xi\in\R^n$, so that the Cauchy-Dirichlet problem \eqref{CD} can be re-written as
\begin{align}\label{CD-v}
\left\{
\begin{array}{ll}
\partial_t v-\nabla \cdot A(v,\nabla v^\beta)= f  &\quad \text{ in } \Omega_T,\\[5pt]
v=0 & \quad \text{ on } \partial\Omega\times(0,T),\\[5pt]
v(\cdot,0)=\psi-z & \quad \text{ in } \overline\Omega,
\end{array}
\right.
\end{align}
We arrive at the following definition of weak solutions by multiplying \eqref{reformulated} by a smooth test function and integrating formally by parts.

\begin{defin}\label{weakdef}
Suppose that $f$, $z$ and $\psi$ satisfy \eqref{assumption:data}. A function $u\colon\Omega_T\to \R $ is a weak solution to the equation \eqref{DNPE} if and only if $v:=u-z$ is non-negative, 
$v^\beta \in L^p(0,T;W^{1,p}(\Omega))$ and 
\begin{align}\label{weakform}
&\iint_{\Omega_T} \big[A(v,\nabla v^\beta)\cdot \nabla \varphi - v\partial_t \varphi\big]\d x\d t=\iint_{\Omega_T} f\varphi \d x\d t,
\end{align}
for all $\varphi \in C^\infty_0(\Omega_T)$. Moreover, in this case we say that $v$ is a solution to \eqref{reformulated}. If in adddition $v\in C([0,T];L^{\beta+1}(\Omega))$, $v^\beta \in L^p(0,T;W^{1,p}_0(\Omega))$ and if $v(0)=\psi -z$, we say that $v$ is a solution to the Cauchy-Dirichlet problem \eqref{CD-v} and $u$ is a solution to the Cauchy-Dirichlet problem \eqref{CD}.
\end{defin}
Now we can state our main existence theorem.

\begin{theo}\label{thm:existence}
Suppose that $f$, $z$ and $\psi$ satisfy \eqref{assumption:data}. Then there exists a weak solution to the Cauchy-Dirichlet problem \eqref{CD} in the sense of Definition~\ref{weakdef}. 
\end{theo}

The main difficulty in the proof of Theorem~\ref{thm:existence} stems from the degeneracy of the shallow medium equation with respect to the solution and from the topography $z$ which in a certain sense plays the role of an obstacle. For the construction of our solution we proceed by an approximation argument. More precisely, we truncate the vector field $A$ with respect to $v$ from above and from below. For $k>1$ we let
\begin{align*}
A_k(v,\xi):=
A\big(T_k(v),\beta T_k(v)^{\beta-1}\xi\big)
\end{align*}
for $v\in \R$ and $\xi \in \R^n$, where the truncation $T_k:\R\to \R$ is given by
\begin{align*}
T_k(s):= \min\big\{ k, \max\{ s, \tfrac1k\}\big\}.
\end{align*}
Then, we consider approximating solutions $v_k$ to a Cauchy-Dirichlet problem associated with the vector fields $A_k$ and with initial and lateral boundary data $\frac1k$, respectively $\frac1k+\Psi$. The next crucial step is to investigate the properties of the approximating solutions. We prove the lower bound $v_k\ge\frac1k$ and an upper bound for $v_k$ independent of $k$. For large $k$ this allows us to conclude that 
\begin{align*}
	A_k(v_k,\nabla v_k)=A\big(v_k,\nabla v_k^\beta\big).
\end{align*}
Next, we prove a uniform bound for the $L^p(\Omega_T)$-norms of the gradients $\nabla v_k^\beta$. This shows that there exists a weakly convergent subsequence of $v_k^\beta$ in $L^p(0,T;W^{1,p}(\Omega))$ to some function $w$. By a delicate compactness argument (see Section~\ref{sec:compactness}) we identify the pointwise limit of $v_k$ as $v:=w^{\frac1\beta}$. In the next step we use the differential equations for $v_k$, the ellipticity of the vector field $A$ and the previously proved convergences to conclude pointwise convergence of the gradients $\nabla v_{k}^\beta\to \nabla v^\beta$ for a subsequence. This finally allows us to pass to the limit in the differential equations for $v_k$ and in turn to conclude that $u=v+z$ is the desired weak solution to the shallow medium equation.

\section{Preliminaries}
\label{Preliminaries}

Here we introduce some notation and present auxiliary tools that will be useful in the course of the paper.

\subsection{Notation}
With $B_\rho(y_o)$ we denote the open ball in $\R^m$, $m\in\N$ with radius $\rho$ at center $y_o\in\R^m$. 
%
For $v,w \geq 0$ we define 
\begin{align}
\label{definition:F}
\b[v,w]&:= \tfrac{1}{\beta+1} (v^{\beta+1}-w^{\beta+1})- w^\beta (v-w)
\\
\notag &\phantom{:} =  \tfrac{\beta}{\beta+1} (w^{\beta+1}-v^{\beta+1})- v (w^\beta-v^\beta),
\end{align}
where $\beta$ is introduced in \eqref{def:beta}.  For convenience we will sometimes use the short hand notation $w(t)=w(\cdot,t)$ for $t\in [0,T]$. 
By $c$ we denote a generic constant which may change from line to line. 

\subsection{Mollifications in time}

Since weak solutions are not necessarily weakly differentiable with respect to time, we will use some mollification techniques. We first recall the definition of Steklov averages. Given a function $v\in L^1(\Omega_T)$ and $h\in(0,T)$, we define its Steklov-mean by 
\begin{equation}\label{def:steklov}
	[v]_h(x,t)	:= 	\frac{1}{h} \int_t^{t+h} v(x,s)\d s, \quad (x,t)\in \Omega \times (0,T-h),
\end{equation}
and the reversed Steklov-mean by
\begin{equation}\label{def:steklov-rev}
	[v]_{\bar h}(x,t) := \frac{1}{h} \int_{t-h}^{t} v(x,s)\d s \quad (x,t) \in \Omega \times (h,T).
\end{equation}
For some basic properties of Steklov averages we refer to Chapter~1.3 of \cite{DiBene} and Section~9 of \cite{DuMi}. For porous medium type equations and more generally doubly nonlinear equations it is often necessary to work with the following so-called exponential time mollification, cf.~\cite{KiLi}. For a function $v\in L^1(\Omega_T)$ and $h\in (0,T)$ we set
\begin{align}
\label{def:moll}
\Ex{v}(x,t):=\frac{1}{h}\int^t_0 e^\frac{s-t}{h}v(x,s)\d s.
\end{align}
for $x\in\Omega$ and $t\in [0,T]$. Moreover, we define the reversed analogue by
\begin{align*}
\Exx{v}(x,t) :=\frac{1}{h}\int^T_t e^\frac{t-s}{h}v(x,s)\d s.
\end{align*}
For details regarding the properties of the exponential mollification we refer to \cite[Lemma 2.2]{KiLi},  \cite[Lemma 2.2]{BoeDuMa}, and \cite[Lemma 2.9]{St}. The properties of the mollification that we will use have been collected for convenience into the following lemma:

\begin{lem}\label{expmolproperties} 
Suppose that $v \in L^1(\Omega_T)$, and let $p\in[1,\infty)$. Then the mollification $\Ex{v}$ defined in \eqref{def:moll} has the following properties:
\begin{enumerate}
\item[(i)] 
If $v\in L^p(\Omega_T)$ then $\Ex{v}\in L^p(\Omega_T)$, 
$$
\norm{\Ex{v}}_{L^p(\Omega_T)}\leq \norm{v}_{L^p(\Omega_T)},
$$
 and $\Ex{v}\to v$ in $L^p(\Omega_T)$. A similar statement holds for $\Exx{v}$.
\item[(ii)]
If $v\in L^p(\Omega_T)$, then $\Ex{v}$ and $\Exx{v}$ have weak time derivatives belonging to $L^p(\Omega_T)$ given by
\begin{align*}
\partial_t \Ex{v}=\tfrac{1}{h}(v-\Ex{v}),\hspace{5mm} \partial_t \Exx{v}=\tfrac{1}{h}(\Exx{v}-v).
\end{align*}
\item[(iii)] 
If $v\in L^p(0,T;W^{1,p}(\Omega))$ then $\Ex{v}\in L^p(0,T;W^{1,p}(\Omega))$ and $\Ex{v}\to v$ in 
$L^p(0,T;W^{1,p}(\Omega))$ as $h\to 0$. Moreover, if $v\in L^p(0,T;W_0^{1,p}(\Omega))$ then $\Ex{v}\in L^p(0,T;W_0^{1,p}(\Omega))$. Similar statements hold for $\Exx{v}$.
\item[(iv)] If $v\in L^p(0,T;L^{p}(\Omega))$ then $\Ex{v},\,\Exx{v} \in C([0,T];L^{p}(\Omega))$.
\end{enumerate}
\end{lem}

%

\subsection{Elementary inequalities}

We now recall some elementary inequalities that will be used later, and start by some useful estimates for the quantity $\b[v,w]$ that was defined in \eqref{definition:F}. The proof can be found in \cite[Lemma 2.2 and Lemma 2.3]{BoDuKoSc}.

\begin{lem}\label{estimates:boundary_terms}
Let $v,w \geq 0$ and $\beta> 1$. Then there exists a constant $c\ge 1$ depending only on $\beta$ such that:
\begin{enumerate}
\item[(i)] $\tfrac 1 c\big| w^{\frac{\beta+1}{2}}-v^{\frac{\beta+1}{2}} \big|^2 \leq \b[v,w] \leq c  \big| w^{\frac{\beta+1}{2}}-v^{\frac{\beta+1}{2}} \big|^2$, \vspace{2mm}
\item[(ii)] $\b[v,w] \geq \tfrac1c |v-w|^{\beta+1}$,
\item[(iii)]$\b[v,w] \leq c |v^{\beta}-w^{\beta}|^{\frac{\beta+1}{\beta}}$.
\end{enumerate}
\end{lem}

The following algebraic lemma is standard and can be obtained by reasoning as in the proof of Lemma 5.1 in Chapter IX of \cite{DiBene}.
\begin{lem}\label{lem:monotone}
Let $p>1$. Then, for any $\xi,\eta \in \R^n$ which in the case $p\in(1,2)$ are not both zero there holds 
\begin{equation*}
	\big(|\xi|^{p-2}\xi - |\eta|^{p-2}\eta\big)\cdot(\xi-\eta)
	\ge
	c\big(|\xi|^2 + |\eta|^2\big)^{\frac{p-2}{2}}|\xi-\eta|^2,
\end{equation*}
where $c$ is a constant depending only on $p$. 
\end{lem}

\begin{rem}\label{rem:monotone}
Let $p>1$ and $\xi,\eta,\zeta \in \R^n$ with either $\xi\not=\zeta$ or $\eta\not=\zeta$ in the case $p\in(1,2)$. Then, by Lemma~\ref{lem:monotone} there holds
\begin{align*}
	&\big(|\xi+\zeta|^{p-2}(\xi+\zeta) - |\eta+\zeta|^{p-2}(\eta+\zeta)\big)
	\cdot(\xi-\eta) \\
	&\qquad=
	\big(|\xi+\zeta|^{p-2}(\xi+\zeta) - |\eta+\zeta|^{p-2}(\eta+\zeta)\big)
	\cdot\big((\xi+\zeta)-(\eta+\zeta)\big) \\
	&\qquad\ge
	\tfrac1c
	\big(|\xi+\zeta|^2 + |\eta+\zeta|^2\big)^{\frac{p-2}{2}}|\xi-\eta|^2 \\
	&\qquad\ge
	\tfrac1c
	\big(|\xi+\zeta| + |\eta+\zeta|\big)^{p-2}|\xi-\eta|^2.
\end{align*}
If $p\ge 2$, we use $|\xi-\eta|^2\le 2(|\xi+\zeta|^2+|\eta+\zeta|^2)$ to conclude 
\begin{align*}
	\big(|\xi+\zeta|^{p-2}(\xi+\zeta) - |\eta+\zeta|^{p-2}(\eta+\zeta)\big)
	\cdot(\xi-\eta) 
	\ge
	c\,|\xi-\eta|^p.
\end{align*}
\end{rem}

\begin{lem}\label{lem:V}
Let $p>1$. Then, for any $\xi,\eta \in \R^n$ there holds 
\begin{equation*}
	|\xi+\eta|^{p-2}(\xi+\eta)\cdot\xi
	\ge
	2^{-p}|\xi|^p - 2^p|\eta|^p
\end{equation*}
\end{lem}

\begin{proof}[Proof]
By Young's inequality we obtain
\begin{align*}
	|\xi+\eta|^{p-2}(\xi+\eta)\cdot\xi 
	&= 
	|\xi+\eta|^{p} - |\xi+\eta|^{p-2}(\xi+\eta)\cdot\eta \\
	&\geq 
	\tfrac12|\xi+\eta|^{p} - 2^{p-1}|\eta|^p \\
	&\geq 
	2^{-p}|\xi|^p - 2^p|\eta|^p.
\end{align*}
This is the claimed inequality.
\end{proof}

\begin{lem}\label{p-laplace-estim}
Let $p>1$ and $\xi,\eta \in \R^n$. In the case $p<2$ we also assume that either $\xi$ or $\eta$ is nonzero. Then there exists a constant $c$ depending only on $p$ such that:
\begin{align*}
\big||\xi|^{p-2}\xi-|\eta|^{p-2}\eta\big|\leq c (|\eta|+ |\xi-\eta|)^{p-2}|\xi-\eta|.
\end{align*}
\end{lem}
\begin{proof}[Proof]
Applying \cite[Lemma~2.2]{AcFu-1989} with $\mu=0$ we obtain
\begin{align*}
\big||\xi|^{p-2}\xi-|\eta|^{p-2}\eta\big|
&\leq c \big(|\eta|^2+ |\xi|^2\big)^{\frac{p-2}{2}}|\xi-\eta|\\
&\leq c (|\eta|+ |\xi|)^{p-2}|\xi-\eta|\\
&\leq 2^{|p-2|}c (|\eta|+ |\xi-\eta|)^{p-2}|\xi-\eta|,
\end{align*}
which is the desired estimate.
\end{proof}

\subsection{Auxiliary tools}

The following lemma can be proven using an inductive argument, see for example \cite[\S\,I,~Lemma~4.1]{DiBene}. 
\begin{lem}
\label{fastconvg} Let $(Y_j)^\infty_{j=0}$ be a sequence of positive real numbers such that 
\begin{equation*}
Y_{j+1}\leq C b^j Y^{1+\delta}_j,
\end{equation*}
where $C, b >1$ and $\delta>0$. If
\begin{equation*}
Y_0\leq C^{-\frac{1}{\delta}}b^{-\frac{1}{\delta^2}},
\end{equation*}
then $(Y_j)$ converges to zero as $j\to\infty$.
\end{lem}

Next, we recall a well-known parabolic Sobolev inequality, which can be found for example in \cite{DiBene}. 
\begin{lem}
\label{lemma:Gagliardo}
Let $\Omega\subset \R^{n}$, $T>0$ and $\theta>0$. Suppose that $q>0$, $p>1$. Then for every 
$$
u\in L^\infty \big(0,T;L^q(\Omega)\big) \cap L^p\big(0,T;W^{1,p}_0(\Omega)\big)
$$
we have
\begin{align*}
\iint_{\Omega_T} |u|^{p(1+\frac qn)} \d x\d t& \leq c\bigg[\esssup_{t\in (0,T)} \int_{\Omega\times \{t\}} |u|^q \d x\bigg]^{\frac pn} \iint_{\Omega_T}  |\nabla u|^p  \d x\d t
\end{align*}
for a constant $c=c(n,p,q,\Omega)$. 
\end{lem}

We now present a version of the Cauchy-Peano existence theorem which will be used in connection with Galerkin's method.
\begin{lem}\label{lem:Cauchy-Peano}
Let $A$ be an invertible $m\times m$ matrix, let $g:\R^m\to \R^m$ be continuous and let $\tilde{f}:[0,T]\to \R^m$ be integrable. Let $y_o\in \R^m$. Then there exists a $\delta >0$ and an absolutely continuous function $y:[0,\delta]\to \R^m$ such that
\begin{align*}
A y'(t) &=g(y(t))+\tilde{f}(t), \hspace{5mm} \textrm{a.e. in }[0,\delta],
\\
y(0)&=y_o.
\end{align*}
\end{lem}
\begin{proof}[Proof]
Since the differential equation is equivalent to 
\begin{align*}
y'(t)= A^{-1}\circ g(y(t))+ A^{-1}\circ \tilde{f}(t),
\end{align*}
and since $A^{-1}\circ g$ is continuous and $A^{-1}\circ \tilde{f}$ is integrable, we may assume $A=I_{m\times m}$.
Fix $r>0$ and set $M:=\max_{\bar{B}_r(y_o)}|g|$. Define $\delta_1:=\tfrac{r}{2M+1}$ and pick $\delta_2>0$ so small that 
\begin{align*}
\int^{\delta_2}_0 |\tilde{f}| \d t < \frac r2.
\end{align*}
We choose $\delta=\min\{\delta_1,\delta_2\}$. Denoting by $c(y_o)$ the constant function $c(y)\equiv y_o$ on $[0,\delta]$ and setting $\widetilde F(t):=M+|\tilde{f}(t)|$, we consider the subset
\begin{align*}
K:=\bigg\{u\in \bar{B}_r(c(y_o))\,:\, |u(t_1)-u(t_2)|\leq \int^{t_2}_{t_1} \widetilde F(t)\d t, \hspace{4mm}0\leq t_1\leq t_2 \leq \delta\bigg\}
\end{align*}
of the vector space $C([0,\delta];\R^m)$ equipped with the sup-norm. The set $K$ is convex and closed. Furthermore, the elements of $K$ are equicontinuous pointwise bounded functions, so the Arzel\`a-Ascoli theorem guarantees that $K$ is compact. We note that we have a well-defined map $T:K\to K$ defined by
\begin{align*}
(T u)(t)= y_o+\int^t_0 \big[g(u(s))+ \tilde{f}(s)\big]\d s.
\end{align*}
Since 
\begin{align*}
|Tu(t)-Tv(t)|\leq \int^t_0 |g(u(s))-g(v(s))|\d s \leq \delta \sup_{s\in [0,\delta]}|g(u(s))-g(v(s))|,
\end{align*}
the uniform continuity of $g$ on the compact set $\bar{B}_r(y_o)$ shows that $T$ is continuous $K\to K$. Thus, Schauder's fixed point theorem guarantees that there exists at least one $y\in K$ such that $Ty=y$. But then $y$ is a solution to the original problem.
\end{proof}

\subsection{Time continuity}
We have taken a certain time continuity as part of the definition of a solution to the Cauchy-Dirichlet problem. 
Our next goal is to show that all weak solutions vanishing on the lateral boundary and having sufficiently high integrability automatically have this type of time continuity. The extra integrability condition $v\in L^{\beta+1}(\Omega_T)$ occuring in the two following lemmas is redundant in the case $\alpha+p\geq 2$, since then $\beta p\geq \beta + 1$. Note that we will only apply the results to the bounded solution found in Section \ref{sec:existence}, which obviously satisfies the extra integrablity condition. 
\begin{lem}\label{lem:weakform_2}
For all $v\in L^{\beta+1}(\Omega_T;\R_{\geq 0})$ with $v^\beta \in L^p(0,T;W^{1,p}_0(\Omega))$ satisfying \eqref{weakform}, all $\zeta \in W^{1,\infty}([0,T];\R_{\geq 0})$ satisfying $\zeta(0)=\zeta(T)=0$, and all $w$ in
\begin{align*}
\mathcal{V}:=\big\{ w \in L^{\beta+1}(\Omega_T)\,|\, w^\beta\in L^p(0,T;W^{1,p}_0(\Omega)), \partial_t w^\beta \in L^{\frac{\beta+1}{\beta}}(\Omega_T) \big\},
\end{align*}
we have
\begin{align}\label{ske}
\iint_{\Omega_T} & \zeta' \b[v,w]\d x \d t\notag\\ &= 
\iint_{\Omega_T} \zeta\Big[\partial_t w^\beta(v-w) + A(v,\nabla v^\beta)\cdot (\nabla v^\beta-\nabla w^\beta) - f (v^\beta-w^\beta)\Big] \d x\d t.
\end{align}
\end{lem}
\begin{proof}[Proof]
We use \eqref{weakform} with the test function $\varphi_h=\zeta(w^\beta-\Ex{v^\beta})$. This is possible since we can find $\varphi_{h,j}\in C^\infty_0(\Omega_T)$ such that $\varphi_{h,j}\to \varphi_h$ in $L^p(0,T;W^{1,p}(\Omega))$ and $\partial_t \varphi_{h,j} \to \partial_t \varphi_h$ in $L^\frac{\beta+1}{\beta}(\Omega_T)$. We see immediately that
\begin{align}\label{simple_h_lims}
\lim_{h\downarrow 0} \iint_{\Omega_T}A(v,\nabla v^\beta) \cdot \nabla \varphi_h \d x \d t = 
\iint_{\Omega_T}A(v,\nabla v^\beta) \cdot (\nabla w^\beta -\nabla v^\beta)\zeta  \d x \d t,
\end{align}
and
\begin{align}\label{simple_h_lims_f}
\lim_{h\downarrow 0} \iint_{\Omega_T}f \varphi_h \d x \d t = 
\iint_{\Omega_T} f \zeta (w^\beta-v^\beta)\d x\d t.
\end{align}
The parabolic term is treated as follows.
\begin{align*}
\iint_{\Omega_T} & v\partial_t \varphi_h \d x\d t 
\\
&= \iint_{\Omega_T} \zeta' v (w^\beta-\Ex{v^\beta})  + \zeta v(\partial_t w^\beta-\partial_t \Ex{v^\beta})\d x\d t 
\\
&= \iint_{\Omega_T} \zeta v\partial_t w^\beta - \zeta\Ex{v^\beta}^\frac{1}{\beta}\partial_t \Ex{v^\beta} + \zeta'  v (w^\beta-\Ex{v^\beta}) + \zeta (\Ex{v^\beta}^\frac{1}{\beta}-v)\partial_t \Ex{v^\beta}\d x\d t 
\\
&\leq \iint_{\Omega_T} \zeta v \partial_t w^\beta + \zeta' \tfrac{\beta}{\beta+1}\Ex{v^\beta}^\frac{\beta+1}{\beta} + \zeta' v (w^\beta-\Ex{v^\beta}) \d x\d t 
\end{align*}
In the third step we have used Lemma \ref{expmolproperties} (ii) to conclude that the last term is nonpositive, and the second term has been integrated by parts. The limit of the parabolic term as $h\to 0$ exists due to \eqref{simple_h_lims}, \eqref{simple_h_lims_f} and \eqref{weakform}, and satisfies the estimate
\begin{align}\label{parab_term_h_lim}
\notag \lim_{h\downarrow 0} \iint_{\Omega_T}  v\partial_t & \varphi_h \d x\d t \leq  \iint_{\Omega_T} \zeta v \partial_t w^\beta + \zeta' \big(\tfrac{\beta}{\beta+1}v^{\beta+1} +  v (w^\beta-v^\beta) \big) \d x\d t 
\\
\notag &= \iint_{\Omega_T} \zeta (v-w) \partial_t w^\beta + \zeta \tfrac{\beta}{\beta+1}\partial_t (w^\beta)^\frac{\beta+1}{\beta} + \zeta' \big(\tfrac{\beta}{\beta+1}v^{\beta+1} +   v (w^\beta-v^\beta) \big) \d x\d t 
\\
&= \iint_{\Omega_T} \zeta (v-w) \partial_t w^\beta - \zeta' \b[v,w] \d x \d t.
\end{align}
In the last step we integrate the second term by parts. Combining \eqref{simple_h_lims}, \eqref{simple_h_lims_f} and \eqref{parab_term_h_lim} we have verified ``$\leq$" in \eqref{ske}. The reverse inequality follows in a similar way by taking $\varphi=\zeta(w^\beta-\Exx{v^\beta})$ as a test function.
\end{proof}
Using the previous lemma we can now conclude the time continuity. 
\begin{lem}\label{lem:time-cont}
Every $v\in L^{\beta+1}(\Omega_T;\R_{\geq 0})$ with $v^\beta \in L^p(0,T;W^{1,p}_0(\Omega))$ satisfying \eqref{weakform} has a representative in $C([0,T];L^{\beta+1}(\Omega))$. 
\end{lem}
\begin{proof}[Proof]
Take $\psi\in C^\infty(\R;[0,1])$ such that $\psi(t)=1$ for $t\leq \tfrac12 T$, $\psi(t)=0$ for $t>\tfrac34 T$ and $|\psi'|\leq \tfrac8T$. For $\tau \in (0, \tfrac12 T)$  and $\varepsilon>0$ so small that $\tau+\varepsilon < \tfrac12 T$, define
\begin{align*}
\chi^\tau_\varepsilon(t)=
\begin{cases}
0, & t<\tau \\
\frac1\varepsilon(t-\tau), & t\in [\tau, \tau+\varepsilon] \\
1, & t> \tau+\varepsilon.
\end{cases}
\end{align*}
Using identity \eqref{ske} from Lemma~\ref{lem:weakform_2} with $\zeta=\chi^\tau_\varepsilon\psi$ and $w=(\Exx{v^\beta})^\frac{1}{\beta}$ we obtain
\begin{align*}
\varepsilon^{-1}\int^{\tau+\varepsilon}_\tau \int_\Omega \b[v,w]\d x \d t &=  \iint_{\Omega_T}\big[A(v,\nabla v^\beta)\cdot (\nabla v^\beta-\nabla w^\beta) + (v-w)\partial_t w^\beta \big]\zeta \d x \d t
\\
&\quad +\iint_{\Omega_T} f (w^\beta-v^\beta)\zeta \d x\d t - \iint_{\Omega_T}\b[v,w]\psi' \d x \d t
\\
&\leq \iint_{\Omega_T}\big[|A(v,\nabla v^\beta)||\nabla v^\beta-\nabla \Exx{v^\beta} | + |f||\Exx{v^\beta}-v^\beta|\big] \d x \d t
\\
&\quad+\frac{c}{T}\iint_{\Omega_T} |v^\beta - \Exx{v^\beta} |^\frac{\beta+1}{\beta}\d x \d t.
\end{align*}
Here we used Lemma \ref{expmolproperties} (ii) to conclude that the term involving $\partial_t w^\beta$ is nonpositive. Also, we have made use of Lemma \ref{estimates:boundary_terms} (iii) to estimate $\b[v,w]$. Passing to the limit $\varepsilon\to 0$ we see that 
\begin{align*}
\int_\Omega \b[v,w](\cdot,\tau) \d x &\leq \iint_{\Omega_T}\big[|A(v,\nabla v^\beta)||\nabla v^\beta-\nabla \Exx{v^\beta} | + |f||\Exx{v^\beta}-v^\beta|\big] \d x \d t
\\
&\quad+\frac{c}{T}\iint_{\Omega_T} |v^\beta - \Exx{v^\beta} |^\frac{\beta+1}{\beta}\d x \d t
\end{align*}
for all $\tau\in (0,\tfrac{T}{2})\setminus N_h$ where $N_h$ is a set of measure zero. Using Lemma \ref{estimates:boundary_terms} (ii) we find a lower bound for the integrand on the left-hand side:
\begin{align*}
|v-w|^{\beta+1} \leq c\, \b[v,w].
\end{align*}
Taking now a sequence $h_j\downarrow 0$ and setting $w_j:= ( [\![v^\beta]\!]_{\bar h_j})^\frac{1}{\beta}$ and $N:=\cup N_{h_j}$ we see that
\begin{align}\label{unilim}
\lim_{j\to \infty} \sup_{t\in [0,\frac{T}{2}]\setminus N} \int_{\Omega} |v-w_j|^{\beta+1}(\cdot,t) \d x = 0
\end{align}
Observe that every map of the form $w=(\Exx{v^\beta})^\frac{1}{\beta}$ is continuous $[0,T]\to L^{\beta+1}(\Omega)$ since
\begin{align*}
|w(x,s)-w(x,t)|^{\beta+1}\leq |w^\beta (x,s)-w^\beta (x,t)|^\frac{\beta+1}{\beta}=|\Exx{v^\beta}(x,s) - \Exx{v^\beta}(x,t)|^\frac{\beta+1}{\beta},
\end{align*}
and $\Exx{v^\beta}$ is continuous $[0,T]\to L^\frac{\beta+1}{\beta}(\Omega)$ by Lemma \ref{expmolproperties} (iv). Because of the uniform limit \eqref{unilim}, $v$ has a representative which is continuous on $[0,\tfrac{T}{2}]\setminus N$ and since $N$ has measure zero we find a representative which is continuous on $[0,\tfrac{T}{2}]$. The continuity on $[\tfrac{T}{2},T]$ follows from a similar argument, but taking $w=(\Ex{v^\beta})^\frac{1}{\beta}$ and reflecting $\chi^\tau_\varepsilon$ and $\psi$ on the interval $[0,T]$.
\end{proof}

\section{Compactness results}\label{sec:compactness}

In this section we provide some compactness results which could be of interest in their own. They are consequences of the well known compactness results in \cite{JacquesSimon}. Such results are necessary when dealing with porous medium type equations and doubly nonlinear equations. For the purpose of this paper we only need Corollary~\ref{thm:JS-th6m}. Since there is not much extra effort to obtain other versions we preferred to include them as well for possible future applications. Some results in this direction were previously obtained for uniformly bounded functions that are piecewise constant in time, see \cite[Lemma~8]{Juengel} and the references therein.

For a family of functions $F \subset L^{1}(0,T;L^1(\Omega;\R^N))$ with $N\ge 1$ we denote 
$$
	F^m:=\big\{|f|^{m-1}f:f\in F\big\} \quad\mbox{for $m>0$.} 
$$
We start with an Arz\'ela-Ascoli type theorem. It will be the basis for all the other compactness results.
For a map $f:[0,T]\to X$ we define the translated function $\tau_h f:[0,T-h]\to X$ by $\tau_h f(t)=f(t+h)$. We can now formulate the following theorem \cite[\S\,3,~Theorem~1]{JacquesSimon}: 

\begin{theo}\label{thm:JS-th1}
Let $1\le p\le\infty$, $T>0$ and $B$ be a Banach space. Let $F\subset L^p(0,T;B)$. Then $F$ is relatively compact in $L^p(0,T;B)$ for $1\le p<\infty$, or in $C([0,T];B)$ for $p=\infty$ if and only if 
\begin{equation}\label{JS_1}
	\bigg\{\int_{t_1}^{t_2} f(t) \,\d t: f\in F\bigg\}
	\quad\mbox{is relatively compact in $B$,} 
	\quad\forall\,0<t_1<t_2<T
\end{equation}
and 
\begin{equation}\label{JS_2}
	\|\tau_h f - f\|_{L^p(0,T-h;B)}\to 0
	\quad\mbox{as $h\downarrow 0$,}
	\quad\mbox{uniformly for $f\in F$.}
\end{equation}
\end{theo}

%

Using Theorem \ref{thm:JS-th1} we obtain the following result.
\begin{theo}\label{thm:JS-th3m}
Let $m \in(0,\infty)$, $1\le p\le\infty$, $1\le q\le\infty$. Define $\theta:=\max\{1, m p\}$ and $\mu:=\max\{1,mq\}$. Let $T>0$ and let $\Omega\subset\R^n$ be a bounded domain. Let $X$ be a Banach space with compact embedding $X\hookrightarrow L^q(\Omega;\R^N)$. Moreover, let $F \subset L^{\theta}(0,T;L^\mu(\Omega;\R^N))$ such that $F^m\subset L^1(0,T;X)$. We assume that 
\begin{equation}\label{JS_3m_0}
	\mbox{$F^m$ is bounded in $L^{p}(0,T;L^q(\Omega;\R^N))$ if $m\ge 1$}
\end{equation}
and
\begin{equation}\label{JS_3m_1}
	\mbox{$F^m$ is bounded in $L^1(0,T;X)$}
\end{equation}
and 
\begin{equation}\label{JS_3m_2}
	\|\tau_h f - f\|_{L^\theta(0,T-h;L^\mu(\Omega;\R^N))}\to 0
	\quad\mbox{as $h\downarrow 0$,}
	\quad\mbox{uniformly for $f\in F$.}
\end{equation}
Then $F^m$ is relatively compact in $L^{p}(0,T;L^q(\Omega;\R^N))$ for $1\le p<\infty$. For $p=\infty$ it is relatively compact in $C([0,T];L^q(\Omega;\R^N))$. 
\end{theo}
Before proceeding with the proof, we note that the assumption $F \subset L^{\theta}(0,T;L^\mu(\Omega;\R^N))$ guarantees that $F^m\subset L^p(0,T;L^q(\Omega;\R^N))$. In the case $m\ge1$ we additionally need to assume that $F^m$ is bounded in $L^p(0,T;L^q(\Omega;\R^N))$. The assumption $F^m\subset L^1(0,T;X)$ is to be understood in the following sense: For every $f\in F$, there is an element of $L^1(0,T;X)$ whose composition with the inclusion $X\hookrightarrow L^q(\Omega)$ coincides with $f^m\in L^p(0,T;L^q(\Omega;\R^N))$.

\begin{proof}[Proof of Theorem \ref{thm:JS-th3m}]
Our aim is to apply Theorem~\ref{thm:JS-th1} to the functions $F^m$ with $B=L^q(\Omega;\R^N)$. From \eqref{JS_3m_1} we know that $F^m$ is bounded in $L^1(t_1,t_2;X)$ for any $0<t_1<t_2<T$. Due to the compact embedding $X\subset L^q(\Omega;\R^N)$ this implies that $\{\int_{t_1}^{t_2} f^m(t) \,\d t: f^m\in F^m\}$ is relatively compact in $L^q(\Omega;\R^N)$ for any $0<t_1<t_2<T$. Hence \eqref{JS_1} is satisfied. Next, we verify assumption \eqref{JS_2}. Suppose first $p<\infty$. If $m\in(0,1)$, we have
\begin{align*}
	\|\tau_h f^m - f^m\|^p_{L^p(0,T-h;L^q(\Omega;\R^N))}
	&=
	\int_0^{T-h} 
	\bigg[\int_\Omega |f^m(t+h)-f^m(t)|^q \,\d x
	\bigg]^{\frac{p}{q}} \d t\\
	&\le
	c\int_0^{T-h} 
	\bigg[\int_\Omega |f(t+h)-f(t)|^{mq} \,\d x
	\bigg]^{\frac{p}{q}} \d t\\
	&\le
	c\int_0^{T-h} 
	\bigg[\int_\Omega |f(t+h)-f(t)|^{\mu} \,\d x
	\bigg]^{\frac{m p}{\mu}} \d t \\
	&\le
	c\,\|\tau_h f - f\|^{mp}_{L^\theta(0,T-h;L^\mu(\Omega;\R^N))}.
\end{align*}
Otherwise, if $m\ge 1$, then $\theta=mp$ and $\mu=mq$ and we compute
\begin{align*}
	&\|\tau_h f^m  - f^m\|^p_{L^p(0,T-h;L^q(\Omega;\R^N))}
	=
	\int_0^{T-h} 
	\bigg[\int_\Omega |f^m(t+h)-f^m(t)|^q \,\d x
	\bigg]^{\frac{p}{q}} \d t\\
	&\qquad\le
	c\int_0^{T-h} 
	\bigg[\int_\Omega |f(t+h)-f(t)|^q
	\big(|f(t+h)| + |f(t)|\big)^{(m-1)q} \,\d x
	\bigg]^{\frac{p}{q}} \d t\\
	&\qquad\le
	c\int_0^{T-h} 
	\bigg[\int_\Omega |f(t+h)-f(t)|^{\mu} \,\d x
	\bigg]^{\frac{p}{\mu}} 
	\bigg[\int_\Omega\big(|f(t+h)| + |f(t)|\big)^{mq} \,\d x
	\bigg]^{\frac{(m-1)p}{mq}}\d t \\
	&\qquad\le
	c\,\|\tau_h f - f\|^{\frac{\theta}{m}}_{L^\theta(0,T-h;L^\mu(\Omega;\R^N))} 
	\|f^m\|^{(1-\frac{1}{m})p}_{L^p(0,T-h;L^q(\Omega;\R^N))}.
\end{align*}
This ensures that also \eqref{JS_2} is satisfied. In the case $p=\infty$, one must again consider the two ranges for $m$. The calculations are similar. Therefore, Theorem~\ref{thm:JS-th1} yields the compactness of $F^m$ in $L^{p}(0,T;L^q(\Omega;\R^N))$ for $1\le p<\infty$, or in $C(0,T;L^q(\Omega;\R^N))$ for $p=\infty$. 
\end{proof} 

This interpolation lemma is from \cite[\S\,8,~Lemma~8]{JacquesSimon}.

\begin{lem}\label{lem:interpolation}
Let $B, X,Y$ be Banach spaces with $X\subset B\subset Y$ and compact embedding $X \hookrightarrow B$. Then, for any $\eta>0$ there exists $M_\eta>0$ such that 
$$
	\|v\|_B
	\le
	\eta\|v\|_X + M_\eta\|v\|_Y
	\quad\mbox{for any $v\in X$.}
$$
\end{lem}

Let $X$ be a Banach space and let $q,\mu \in [1,\infty)$. Throughout the rest of this section we will consider compact embeddings $T:X\to L^q(\Omega;\R^N)$ and $S:L^{\mu'}(\Omega;\R^N)\to X'$ that are compatible in the following way. For any $v\in X$ for which it happens that the function $Tv$ belongs to $L^\mu(\Omega;\R^N)$ and for any $w\in L^{\mu'}(\Omega;\R^N)$ we have 
\begin{align}\label{operator_condition}
\int_\Omega Tv \cdot w\d x = \langle v, S w\rangle,
\end{align}
where $\langle\cdot,\cdot\rangle$ denotes the dual pairing of $X$ and $X'$.

We need the following more complicated version of Lemma~\ref{lem:interpolation}.


\begin{lem}\label{lem:interpolation-PM}
Let $m\in(0,\infty)$, $p,q,\mu\in[1,\infty)$, $\theta:=\max\{1,m p\}$, $T>0$, and $\Omega\subset\R^n$ be a bounded domain and $X, Y$ be Banach spaces such that $X\subset L^q(\Omega;\R^N)$ and $L^{\mu'}(\Omega;\R^N)\subset X'\subset Y$ with compact embeddings $T: X \hookrightarrow L^q(\Omega;\R^N)$ and $S: L^{\mu'}(\Omega;\R^N)\hookrightarrow X'$ satisfying \eqref{operator_condition}. Then, for any $\eta>0$ there exists $M_\eta>0$ such that 
\begin{align*}
	&\|\tau_h f^m -f^m\|_{L^{p}(0,T-h;L^{q}(\Omega;\R^N))} \\
	&\quad\le
	\| f^m\|_{L^p(0,T;X)}^{\frac{m}{m+1}}
	\bigg[
	\eta \Big[\|f^m\|_{L^{p}(0,T;X)}^{\frac{1}{m+1}} + 
	\|f\|_{L^{\theta}(0,T;L^{\mu'}(\Omega;\R^N))}^{\frac{m}{m+1}} \Big] 
	+ 
	M_{\eta} \|\tau_h f-f\|_{L^{\theta}(0,T-h;Y)}^{\frac{m}{m+1}} 
	\bigg].
\end{align*}
for any $f\in L^{\theta}(0,T;L^{\mu'}(\Omega;\R^N))$ with $f^m\in L^p(0,T;X)\cap L^p(0,T;L^\mu(\Omega;\R^N))$.

Note that the assumption $f^m\in L^p(0,T;X)$ must be understood in the same sense as in Lemma \ref{thm:JS-th3m}. In the sequel we do not explicitly write out the embeddings. It will be clear from the context when they should be present.
\end{lem}

\begin{proof}[Proof]
We consider $v\in L^p(0,\tau;X)\cap L^p(0,\tau;L^\mu(\Omega;\R^N))$ (i.e. the composition of $v$ with the embedding $T$ results in a function belonging to $L^p(0,\tau;L^\mu(\Omega;\R^N))$) and $w\in L^{\theta}(0,\tau;L^{\mu'}(\Omega;\R^N))$ with $\tau\in(0,T]$. Since $L^{\mu'}(\Omega;\R^N)\subset X'$, we know that $\langle v(t),w(t)\rangle$ is defined for a.e.~$t\in(0,\tau)$, where $\langle\cdot,\cdot\rangle$ denotes the dual pairing of $X$ and $X'$. Hence, we may calculate
\begin{align}\label{identify}
	\bigg[\int_0^{\tau} \bigg|\int_\Omega
	v \cdot w \,\d x \bigg|^{\frac{\theta p}{\theta+p}} \d t \bigg]^{\frac1p}
	&=
	\bigg[\int_0^{\tau} 
	|\langle v, w\rangle |^{\frac{\theta p}{\theta+p}}
	\,\d t \bigg]^{\frac1p} \nonumber\\
	&\le
	\bigg[\int_0^{\tau} \| v\|_X^{\frac{\theta p}{\theta+p}} 
	\|w\|_{X'}^{\frac{\theta p}{\theta+p}}
	\,\d t \bigg]^{\frac1p}  \nonumber\\
	&\le
	\| v\|_{L^p(0,\tau;X)}^{\frac{\theta}{\theta+p}} 
	\|w\|_{L^\theta(0,\tau;X')}^{\frac{\theta}{\theta+p}} ,
\end{align}
where in the last line we applied H\"older's inequality. Here we have omitted the embeddings $S$ and $T$ to simplify the notation. In particular, choosing $v=f^m$ and $w=f$ in \eqref{identify} yields
\begin{align}\label{f_beta+1}
	\bigg[\int_0^{T} \bigg[\int_\Omega
	|f|^{m+1} \,\d x \bigg]^{\frac{\theta p}{\theta+p}} \d t \bigg]^{\frac1p}
	&\le
	\|f^m\|_{L^p(0,T;X)}^{\frac{\theta}{\theta+p}} 
	\|f\|_{L^\theta(0,T;X')}^{\frac{\theta}{\theta+p}} \nonumber\\
	&\le
	c\, \|f^m\|_{L^p(0,T;X)}^{\frac{\theta}{\theta+p}} 
	\|f\|_{L^\theta(0,T;L^{\mu}(\Omega;\R^N))}^{\frac{\theta}{\theta+p}} .
\end{align}
For $t\in [0,T-h]$, set $\Omega^t=\Omega \cap \{\tau_h f^m(t)\neq f^m(t)\}$. Applying in turn H\"older's inequality, Lemma~\ref{lem:monotone} and \eqref{identify} with the choice $v=\tau_h f^m-f^m$ and $w=\tau_h f-f$ leads to 
\begin{align*}
	\mbox{I}
	&:=
	\bigg[\int_0^{T-h} \bigg|\int_{\Omega^t} 
	\big(|\tau_hf^m| + |f^m|\big)^{\frac{1-m}{m}} 
	|\tau_hf^m-f^m|^{2} \,\d x\bigg|^{\frac{m p}{m+1}} \d t
	\bigg]^{\frac1p} \nonumber\\
	&\le 
	T^{\frac{\theta-m p}{p\theta(m+1)}}
	\bigg[\int_0^{T-h} \bigg|\int_{\Omega^t} 
	\big(|\tau_hf^m| + |f^m|\big)^{\frac{1-m}{m}} 
	\big|\tau_hf^m-f^m\big|^{2} \,\d x\bigg|^{\frac{\theta p}{\theta+p}} \d t
	\bigg]^{\frac1p\cdot \frac{m(\theta+p)}{\theta(m+1)}} \nonumber\\
	&\le
	c\bigg[\int_0^{T-h} \bigg|\int_\Omega
	\big(\tau_h f^m-f^m\big) \cdot (\tau_h f-f) 
	\,\d x \bigg|^{\frac{\theta p}{\theta+p}} \d t 
	\bigg]^{\frac1p\cdot \frac{m(\theta+p)}{\theta(m+1)}} \nonumber\\
	&\le
	c\,\| f^m\|_{L^p(0,T;X)}^{\frac{m}{m+1}} 
	\|\tau_h f-f\|_{L^\theta(0,T-h;X')}^{\frac{m}{m+1}},
\end{align*}
for some constant $c=c(m)\ge 1$. The set $\Omega^t$ was introduced instead of $\Omega$ to avoid dividing by zero in the case $m>1$.
Now, let $\eta_1>0$. By Lemma~\ref{lem:interpolation} there exists $M_{\eta_1}>0$ such that 
\begin{align*}
	\mbox{I}
	&\le
	c\,\| f^m\|_{L^p(0,T;X)}^{\frac{m}{m+1}}
	\bigg[
	\eta_1^{\frac{m}{m+1}}
	\|\tau_h f-f\|_{L^{\theta}(0,T-h;L^{\mu'}(\Omega;\R^N))}^{\frac{m}{m+1}} + 
	M_{\eta_1}^{\frac{m}{m+1}} 
	\|\tau_h f-f\|_{L^{\theta}(0,T-h;Y)}^{\frac{m}{m+1}}
	\bigg] \\
	&\le
	c\,\| f^m\|_{L^p(0,T;X)}^{\frac{m}{m+1}}
	\bigg[
	\eta_1^{\frac{m}{m+1}}
	\|f\|_{L^{\theta}(0,T;L^{\mu'}(\Omega;\R^N))}^{\frac{m}{m+1}} + 
	M_{\eta_1}^{\frac{m}{m+1}} 
	\|\tau_h f-f\|_{L^{\theta}(0,T-h;Y)}^{\frac{m}{m+1}}
	\bigg],
\end{align*}
where $c=c(m,p)$.  
If $m\le 1$ we therefore immediately conclude that 
\begin{align*}
	\|\tau_h f^m & -f^m\|_{L^{p}(0,T-h;L^{\frac{m+1}{m}}(\Omega;\R^N))} 
	\le c\, \mbox{I}\\
	&\le
	c\, \| f^m\|_{L^p(0,T;X)}^{\frac{m}{m+1}}
	\bigg[
	\eta_1^{\frac{m}{m+1}}
	\|f\|_{L^{\theta}(0,T;L^{\mu'}(\Omega;\R^N))}^{\frac{m}{m+1}} + 
	M_{\eta_1}^{\frac{m}{m+1}} 
	\|\tau_h f-f\|_{L^{\theta}(0,T-h;Y)}^{\frac{m}{m+1}} 
	\bigg],
\end{align*}
where $c=c(m,p)$.
If $m> 1$ we use H\"older's inequality to estimate $\mbox{I}$ from below and get with the abbreviation $F:=(|\tau_h f^m| + |f^m|)^{\frac{1}{m}}$, inequality \eqref{f_beta+1} and Young's inequality that
\begin{align*}
	&\|\tau_h f^m -f^m\|_{L^{p}(0,T-h;L^{\frac{m+1}{m}}(\Omega;\R^N))} \\
	&\ =
	\bigg[\int_0^{T-h} \bigg[\int_{\Omega^t} 
	F^{\frac{(m+1)(m-1)}{2m}}F^{\frac{(m+1)(1-m)}{2m}}
	|\tau_h f^m-f^m|^{\frac{m+1}{m}} 
	\,\d x\bigg]^{\frac{m p}{m+1}} \d t \bigg]^\frac1p\\
	&\ \le
	\bigg[\int_0^{T-h} \bigg[\int_\Omega 
	F^{m+1} \,\d x\bigg]^{\frac{m p}{m+1}} \d t
	\bigg]^{\frac{m-1}{2m p}}
	\bigg[\int_0^{T-h} \bigg[\int_{\Omega^t}
	F^{1-m} |\tau_h f^m-f^m|^{2} 
	\,\d x\bigg]^{\frac{m p}{m+1}} \d t \bigg]^{\frac{m+1}{2m p}} \\
	&\ \le
	c \bigg[\int_0^T \bigg[
	\int_\Omega |f|^{m+1} \d x \bigg]^{\frac{m p}{m+1}}\d t
	\bigg]^{\frac{m-1}{2m p}} \cdot 
	\mbox{I}^{\frac{m+1}{2m}} \\
	&\ \le 
	c\, \|f^m\|_{L^p(0,T;X)}^{\frac{m}{m+1}} 
	\|f\|_{L^\theta(0,T;L^{\mu'}(\Omega;\R^N))}^{\frac{m-1}{2(m+1)}} 
	\Big[
	\eta_1^{\frac{1}{2}}
	\|f\|_{L^{\theta}(0,T;L^{\mu'}(\Omega;\R^N))}^{\frac{1}{2}} + 
	M_{\eta_1}^{\frac{1}{2}} 
	\|\tau_h f-f\|_{L^{\theta}(0,T-h;Y)}^{\frac{1}{2}}
	\Big] \\
	&\ \le
	c\, \| f^m\|_{L^p(0,T;X)}^{\frac{m}{m+1}}
	\Big[
	2\eta_1^{\frac{1}{2}}
	\|f\|_{L^{\theta}(0,T;L^{\mu'}(\Omega;\R^N))}^{\frac{m}{m+1}} + 
	\eta_1^{-\frac{(m-1)}{2(m+1)}} M_{\eta_1}^{\frac{m}{m+1}} 
	\|\tau_h f-f\|_{L^{\theta}(0,T-h;Y)}^{\frac{m}{m+1}} 
	\Big].
\end{align*}
We have also used the fact that $\theta=m p$ when $m>1$. Combining both cases and choosing $\eta_1$ in a suitable way, we conclude that for any $\eta_2>0$ there exists $M_{\eta_2}>0$ such that 
\begin{align*}
	\|\tau_h f^m & -f^m\|_{L^{p}(0,T-h;L^{\frac{m+1}{m}}(\Omega;\R^N))} \\
	&\le
	\| f^m\|_{L^p(0,T;X)}^{\frac{m}{m+1}}
	\Big[
	\eta_2
	\|f\|_{L^{\theta}(0,T;L^{\mu'}(\Omega;\R^N))}^{\frac{m}{m+1}} + 
	M_{\eta_2}
	\|\tau_h f-f\|_{L^{\theta}(0,T-h;Y)}^{\frac{m}{m+1}} 
	\Big].
\end{align*}
If $q\le \frac{m+1}{m}$ the asserted inequality follows by an application of H\"older's inequality. In the other case where $q> \frac{m+1}{m}$ we apply Lemma~\ref{lem:interpolation} again and find that for $\eta>0$ there exists $M_{\eta}>0$ such that
\begin{align*}
	\|\tau_h f^m & -f^m\|_{L^{p}(0,T-h;L^{q}(\Omega;\R^N))} \\
	&\le
	\tfrac12 \eta \|\tau_h f^m-f^m\|_{L^{p}(0,T-h;X)} + 
	M_{\eta} \|\tau_h f^m -f^m\|_{L^{p}(0,T-h;L^{\frac{m+1}{m}}(\Omega;\R^N))} \\
	&\le
	\eta \|f^m\|_{L^{p}(0,T;X)} \\
	&\quad+ 
	c M_{\eta}\| f^m\|_{L^p(0,T;X)}^{\frac{m}{m+1}}
	\Big[
	\eta_2
	\|f\|_{L^{\theta}(0,T;L^{\mu'}(\Omega;\R^N))}^{\frac{m}{m+1}} + 
	M_{\eta_2} 
	\|\tau_h f-f\|_{L^{\theta}(0,T-h;Y)}^{\frac{m}{m+1}} 
	\Big].
\end{align*}
At this point the claimed inequality follows by choosing $\eta_2$ so small that $cM_\eta\eta_2\le\eta$. 
\end{proof}

With these prerequisites at hand we are able to prove the following more refined version of Theorem~\ref{thm:JS-th3m}, where assumption \eqref{JS_3m_2} is weakened in the sense that $L^\mu(\Omega;\R^N)$ is replaced by a Banach space $Y$ with $X'\subset Y$.

\begin{theo}\label{thm:JS-th5m}
Let $m\in(0,\infty)$, $p,q,\mu\in[1,\infty)$, $\theta:=\max\{1,m p\}$, $T>0$, and $\Omega\subset\R^n$ be a bounded domain and $X,Y$ be Banach spaces such that $X\subset L^q(\Omega;\R^N)$ and $L^{\mu'}(\Omega;\R^N)\subset X'\subset Y$ with compact embeddings $X \hookrightarrow L^{q}(\Omega;\R^N)$ and $L^{\mu'}(\Omega;\R^N)\hookrightarrow X'$  satisfying \eqref{operator_condition}. 
Moreover, let $F \subset L^{\theta}(0,T;L^{\mu'}(\Omega;\R^N))$ such that $F^m\subset L^p(0,T;X)\cap L^p(0,T;L^\mu(\Omega;\R^N))$. We assume that 
\begin{equation}\label{JS_5m_0a}
	\mbox{$F$ is bounded in $L^\theta(0,T;L^{\mu'}(\Omega;\R^N))$}
\end{equation}
and 
\begin{equation}\label{JS_5m_1a}
	\mbox{$F^m$ is bounded in $L^p(0,T;X)$}.
\end{equation}
\begin{itemize}
\item[(i)]
If 
\begin{equation*}
	\|\tau_h f - f\|_{L^\theta(0,T-h;Y)}\to 0
	\quad\mbox{as $h\downarrow 0$,}
	\quad\mbox{uniformly for $f\in F$}
\end{equation*}
is satisfied, then $F^m$ is relatively compact in $L^{p}(0,T;L^q(\Omega;\R^N))$. 
\item[(ii)]
If 
\begin{equation*}
	\|\tau_h f - f\|_{L^1(0,T-h;Y)}\to 0
	\quad\mbox{as $h\downarrow 0$,}
	\quad\mbox{uniformly for $f\in F$}
\end{equation*}
is satisfied, then $F^m$ is relatively compact in $L^{\tilde p}(0,T;L^q(\Omega;\R^N))$ for any $\tilde p\in[1,p)$. 
\end{itemize}
\end{theo}

\begin{proof}[Proof]
Our aim is to apply Theorem~\ref{thm:JS-th1} with $B=L^q(\Omega;\R^N)$ and $F^m$ instead of $F$. From \eqref{JS_5m_1a} we know that $F^m$ is bounded in $L^p(t_1,t_2;X)$ and hence in $L^1(t_1,t_2;X)$ for any $0<t_1<t_2<T$. Due to the compact embedding $X\hookrightarrow L^{q}(\Omega;\R^N)$ this implies that $\{\int_{t_1}^{t_2} f^m(t) \,\d t: f^m\in F^m\}$ is relatively compact in $L^q(\Omega;\R^N)$ for any $0<t_1<t_2<T$. Hence \eqref{JS_1} is satisfied. Next, we verify assumption \eqref{JS_2}.
For $\eta>0$ we denote by $M_\eta>0$ the constant from Lemma~\ref{lem:interpolation-PM}. The application of the Lemma yields that 
\begin{align*}
	&\|\tau_h f^m - f^m\|_{L^p(0,T-h;L^q(\Omega;\R^N))} \\
	&\quad\le
	\| f^m\|_{L^p(0,T;X)}^{\frac{m}{m+1}}
	\bigg[
	\eta \Big[\|f^m\|_{L^{p}(0,T;X)}^{\frac{1}{m+1}} + 
	\|f\|_{L^{\theta}(0,T;L^{\mu'}(\Omega;\R^N))}^{\frac{m}{m+1}} \Big] 
	+ 
	M_{\eta} \|\tau_h f-f\|_{L^{\theta}(0,T-h;Y)}^{\frac{m}{m+1}} 
	\bigg]
\end{align*}
holds true for any $f\in F$. 
Assumptions \eqref{JS_5m_0a} and \eqref{JS_5m_1a} ensure the existence of a constant $C>0$ such that $\|f^m\|_{L^{p}(0,T;X)}\le C$ and $\|f\|_{L^{\theta}(0,T;L^{\mu'}(\Omega;\R^N))}\le C$ for any $f\in F$. 
Hence, for $\epsilon>0$ we may choose $\eta>0$ in the preceding inequality small enough such that 
\begin{align*}
	\|\tau_h f^m - f^m\|_{L^p(0,T-h;L^q(\Omega;\R^N))} 
	\le
	\epsilon 
	+ 
	C_{\epsilon} \|\tau_h f-f\|_{L^{\theta}(0,T-h;Y)}^{\frac{m}{m+1}}, 
\end{align*}
for a constant $C_\epsilon>0$ depending on $\epsilon$, but not on $h$. If the assumption of assertion\,(i) holds, then there is $h_\varepsilon$ such that if $0<h<h_\varepsilon$, then
\begin{align*}
\|\tau_h f-f\|_{L^{\theta}(0,T-h;Y)}^{\frac{m}{m+1}}< \varepsilon/C_\varepsilon,
\end{align*} 
for all $f\in F$. But this means that 
\begin{align*}
\|\tau_h f^m - f^m\|_{L^p(0,T-h;L^q(\Omega;\R^N))}  < 2\varepsilon
\end{align*}
for all $f\in F$ when $h\in (0,h_\varepsilon)$. Since $\epsilon>0$ was arbitrary this verifies assumption \eqref{JS_2} in Theorem~\ref{thm:JS-th1}. Therefore, the application of the Theorem yields the compactness of $F^m$ in $L^{p}(0,T;L^q(\Omega;\R^N))$ proving assertion\,(i).

The second assertion\,(ii) will follow by an interpolation argument. If $\theta=1$, then the result already follows by (i). Therefore it is enough to consider the case $1<\theta=m p$. We consider $\tilde p\in[1,p)$. Without loss of generality we may assume that $\tilde\theta:=m\tilde p>1$. We interpolate 
\begin{align*}
	\|\tau_h f-f\|_{L^{\tilde\theta}(0,T-h;Y)}
	&\le
	\|\tau_h f-f\|_{L^{\theta}(0,T-h;Y)}^{\frac{\theta(\tilde\theta-1)}{\tilde\theta(\theta-1)}}
	\|\tau_h f-f\|_{L^{1}(0,T-h;Y)}^{\frac{\theta-\tilde\theta}{\tilde\theta(\theta-1)}} \\
	&\le
	c\,\|f\|_{L^{\theta}(0,T;L^{\mu'}(\Omega;\R^N))}^{\frac{\theta(\tilde\theta-1)}{\tilde\theta(\theta-1)}}
	\|\tau_h f-f\|_{L^{1}(0,T-h;Y)}^{\frac{\theta-\tilde\theta}{\tilde\theta(\theta-1)}}.
\end{align*}
Due to Assumption \eqref{JS_5m_0a} $\|f\|_{L^{\theta}(0,T;L^{\mu'}(\Omega;\R^N))}$ is bounded independent of $f$ and therefore  we have that $\|\tau_h f-f\|_{L^{\tilde\theta}(0,T-h;Y)}\to0$ as $h\downarrow 0$ uniformly for $f\in F$. This allows us to apply Theorem~\ref{thm:JS-th1} with $\tilde p$ instead of $p$ and thus yields assertion (ii). 
\end{proof}

Applying Theorem~\ref{thm:JS-th5m} with $X=W^{1,p}_0(\Omega;\R^N)$ yields the following Corollary.

\begin{cor}\label{thm:JS-th6m}
Let $m\in(0,\infty)$, $p\in[1,\infty)$, and $T>0$, $\Omega\subset\R^n$ be a bounded domain and $Y$ be a Banach space such that $(W^{1,p}_0(\Omega;\R^N))'\subset Y$. 
Moreover, let $\theta:=\max\{1,m p\}$, $q\in[p,\frac{np}{n-p})$, $\mu\in[p,\frac{np}{n-p})$ if $p<n$ and $q,\mu\in[p,\infty)$ if $p\ge n$ and consider $F \subset L^{\theta}(0,T;L^{\mu'}(\Omega;\R^N))$ such that $F^m\subset L^p(0,T;W^{1,p}(\Omega;\R^N))$. We assume that 
\begin{equation}\label{JS_5m_0}
	\mbox{$F$ is bounded in $L^\theta(0,T;L^{\mu'}(\Omega;\R^N))$}
\end{equation}
and 
\begin{equation}\label{JS_5m_1}
	\mbox{$F^m$ is bounded in $L^p(0,T;W^{1,p}(\Omega;\R^N))$.}
\end{equation}
\begin{itemize}
\item[(i)]
If 
\begin{equation*}
	\|\tau_h f - f\|_{L^\theta(0,T-h;Y)}\to 0
	\quad\mbox{as $h\downarrow 0$,}
	\quad\mbox{uniformly for $f\in F$}
\end{equation*}
is satisfied, then $F^m$ is relatively compact in $L^{p}(0,T;L^q(\Omega;\R^N))$. 
\item[(ii)]
If 
\begin{equation*}
	\|\tau_h f - f\|_{L^1(0,T-h;Y)}\to 0
	\quad\mbox{as $h\downarrow 0$,}
	\quad\mbox{uniformly for $f\in F$}
\end{equation*}
is satisfied, then $F^m$ is relatively compact in $L^{\tilde p}(0,T;L^q(\Omega;\R^N))$ for any $\tilde p\in[1,p)$. 
\end{itemize}
\end{cor}

\begin{proof}[Proof]
We restrict ourselves to the proof of case (i), since (ii) is completely analogous.  Let $\widetilde\Omega\Subset\Omega$ and $\eta\in C_0^\infty(\Omega)$ be a nonnegative function with $\eta\equiv 1$ in $\widetilde\Omega$. We now consider the family of functions $F_\eta:=\{\eta^{\frac{1}{m}}f: f\in F\}$. Then, $F_\eta\subset L^{\theta}(0,T;L^{\mu'}(\Omega;\R^N))$ and $F_\eta^m \subset L^p(0,T;W^{1,p}_0(\Omega;\R^N))$. Our goal is to apply Theorem~\ref{thm:JS-th5m} with $X=W^{1,p}_0(\Omega;\R^N)$ and $F_\eta$ in place of $F$. Note that due to the parameter ranges, all elements of $X$ are $L^q$-integrable, and the inclusion $T:X \to L^q(\Omega;\R^N)$ is compact. Similarly, we have a compact inclusion $\tilde{T}:X \to L^\mu(\Omega;\R^N)$. Furthermore, we have an embedding $L^{\mu'}(\Omega;\R^N)\hookrightarrow X'$ given by $S:=\tilde{T}'\circ J$, where $\tilde{T}': (L^\mu(\Omega;\R^N))'\to X'$ is the adjoint of $\tilde{T}$, and $J$ is the standard isomorphism from $L^{\mu'}(\Omega;\R^N)$ to $(L^\mu(\Omega;\R^N))'$. Since $\tilde{T}$ is compact, Schauder's theorem guarantees that also $S$ is compact. The condition \eqref{operator_condition} follows directly from the definitions of $S$, $T$ and $\tilde{T}$.
Notice that assumptions \eqref{JS_5m_0}, \eqref{JS_5m_1} and (i) imply the corresponding assumptions in Theorem~\ref{thm:JS-th5m}. Thus all the assumptions of Theorem~\ref{thm:JS-th5m} are satisfied, and hence $F_\eta^m$ is relatively compact in $L^{p}(0,T;L^q(\Omega;\R^N))$. In particular, this implies that $F^m$ is relatively compact in $L^{p}(0,T;L^q(\widetilde\Omega;\R^N))$. Since $\widetilde\Omega\Subset\Omega$ was arbitrary we may conclude the relative compactness of $F^m$ in $L^{p}(0,T;L^q(\Omega;\R^N))$ by a diagonal argument.
\end{proof}


In the previous proof we choose $X$ to consist of compactly supported Sobolev functions since the corresponding space of arbitrary Sobolev functions typically is not compactly embedded into $L^q$. Such a result holds only if $\Omega$ is sufficiently regular, for example if $\Omega$ satisfies a cone condition, see \cite[Theorem 6.3]{Ad}. In that case, the proof is somewhat simpler.

\begin{rem}
The assumption on the uniform convergence of the time differences in Theorem~\ref{thm:JS-th6m}\,(i) and (ii) are satisfied for functions with integrable time derivative. More precisely, \cite[\S\,5,~Lemma~4]{JacquesSimon} ensures that 
\begin{equation*}
	\mbox{$\{\partial_t f: f\in F\}$ is bounded in $L^\theta(0,T;Y)$}
\end{equation*}
implies assumption (i) of Theorem~\ref{thm:JS-th6m} and 
\begin{equation*}
	\mbox{$\{\partial_t f: f\in F\}$ is bounded in $L^1(0,T;Y)$}
\end{equation*}
implies assumption (ii) of Theorem~\ref{thm:JS-th6m}.
\end{rem}

\section{The approximation scheme}
For $k>1$ we approximate the vector field $A$ defined in \eqref{def:A} by vector fields $A_k$ given by
\begin{align}\label{approxfield}
A_k(v,\xi):=
A\big(T_k(v),\beta T_k(v)^{\beta-1}\xi\big)
=
\begin{cases}
A(\frac{1}{k},\beta\frac{1}{k^{\beta-1}}\xi), \hspace{7mm} v<\frac{1}{k}
\\
A(v,\beta v^{\beta-1}\xi), \hspace{7mm} \frac{1}{k}\leq v \leq k
\\
A(k, \beta k^{\beta-1}\xi), \hspace{7mm} v > k,
\end{cases}
\end{align}
for $v\in \R$ and $\xi \in \R^n$, where the truncation $T_k:\R\to \R$ is defined by
\begin{align}\label{truncation}
T_k(s):= \min\big\{ k, \max\{ s, \tfrac1k\}\big\}.
\end{align}
Exploiting the definiton of $A$, we can express $A_k$ as
\begin{align}\label{A_k-alt-expression}
A_k(v,\xi)=
T_k^\alpha(v) |\xi+\nabla z|^{p-2}(\xi+\nabla z) =
\begin{cases}
k^{-\alpha}|\xi+\nabla z|^{p-2}(\xi+\nabla z), \hspace{4mm} & v<\frac{1}{k}
\\
v^\alpha |\xi+\nabla z |^{p-2}(\xi + \nabla z), & \frac{1}{k}\leq v \leq k
\\
k^\alpha|\xi+\nabla z|^{p-2}(\xi+\nabla z), & v > k.
\end{cases}
\end{align}

\subsection{Properties of $A_k$}
Using the properties of the vectorfield corresponding to the $p$-Laplace operator we can verify the following useful basic properties of $A_k$.

\subsubsection{Monotonicity}
Due to Remark~\ref{rem:monotone} we know that there exists a constant $c$ depending only on $p$ such that 
\begin{align}\label{Ak-monotone}
(A_k(v,\xi)-A_k(v,\eta))\cdot (\xi-\eta) 
&\geq \begin{cases}
c\, k^{-\alpha}|\xi-\eta|^p, \hspace{4mm}&p\geq 2
\\[3pt]
c\, k^{-\alpha}\big(|\xi+\nabla z|^2+|\eta+\nabla z|^2\big)^{\frac{p-2}{2}}|\xi-\eta|^2, & p<2,
\end{cases}
\end{align}
holds true for any $v\in\R$ and $\xi,\eta\in\R^n$. If $p<2$ and $\xi\equiv\eta\equiv -\nabla z$, then the right-hand side has to be interpreted as zero. 

\subsubsection{Boundedness}
For any $v\in\R$ and $\xi\in\R^n$ we have
\begin{align}\label{Ak-bd}
|A_k(v,\xi)|\leq k^\alpha |\xi+\nabla z|^{p-1}\leq 2^{p-1}k^\alpha \big(|\xi|^{p-1}+|\nabla z|^{p-1}\big).
\end{align}

\subsubsection{Coercivity}
Due to the definition of $A_k$ and Lemma~\ref{lem:V} we have that
\begin{align}\label{coercivity}
	A_k(v,\xi)\cdot \xi 
	&=
	T_k^\alpha(v) |\xi+\nabla z|^{p-2}(\xi+\nabla z)\cdot \xi \nonumber\\
	&\ge
	T_k^\alpha(v) \big[2^{-p}|\xi|^p - 2^p|\nabla z|^p\big] \nonumber\\
	&\ge
	2^{-p}k^{-\alpha} |\xi|^p - 2^pk^\alpha |\nabla z|^p
\end{align}
for any $v\in\R$ and $\xi\in\R^n$.

\subsection{Weak solutions of the approximating equation}
In this section we want to find weak solutions to the approximating problems 
\begin{align}\label{approxprob}
\left\{
\begin{array}{ll}
\partial_t v_k - \nabla \cdot A_k(v_k, \nabla v_k) = f-k^{-\alpha}\nabla \cdot \big(|\nabla z|^{p-2}\nabla z\big) &\quad \text{ in } \Omega_T,\\[5pt]
v_k = \frac1k & \quad \text{ on } \partial\Omega\times(0,T),\\[5pt]
v_k(\cdot,0) = \frac1k+\Psi & \quad \text{ in } \overline\Omega,
\end{array}
\right.
\end{align}
where $\Psi=\psi-z$. By formally integrating by parts we are led to the following definition.
\begin{defin}\label{approxdef}
A function $v_k\in C([0,T]; L^2(\Omega))\cap \frac{1}{k}+ L^p(0,T; W^{1,p}_0(\Omega))$ is an admissible weak solution to the Cauchy-Dirichlet problem \eqref{approxprob} if 
\begin{align}\label{lampredotto}
\iint_{\Omega_T} \big[A_k(v_k,\nabla v_k)\cdot \nabla \varphi - v_k\partial_t\varphi\big]\d x\d t = 
\iint_{\Omega_T} \big[f\varphi + k^{-\alpha}|\nabla z|^{p-2}\nabla z\cdot \nabla \varphi\big] \d x \d t
\end{align}
for all $\varphi\in C^\infty_0(\Omega_T)$ and $v_k(\cdot,0)=\tfrac1k+\Psi$ in $\Omega$. 
\end{defin}

Using a test function of the form $[\varphi(x,t)\xi(t)]_{\bar{h}}$ where $\varphi \in C^\infty(\bar{\Omega}\times [0,T])$ vanishes if $x$ is outside a compact subset of $\Omega$ and $\xi$ is any smooth function compactly supported in $(0,T-h)$, one easily verifies that any solution $v_k$ in the sense of Definition \ref{approxdef} satisfies the following equation with Steklov-means $[\,\cdot\,]_h$ defined in \eqref{def:steklov}:
\begin{align}\label{trippa}
&\int^b_a\int_\Omega \Big[\partial_t [v_k]_{h} \varphi + \big[A_k(v_k,\nabla v_k)\big]_{h} \cdot \nabla \varphi\Big] \d x \d t\nonumber\\ 
&\qquad\qquad\qquad\qquad= \int^b_a\int_\Omega \big[[f]_{h} \varphi + k^{-\alpha}|\nabla z|^{p-2}\nabla z\cdot\nabla\varphi\big]\d x \d t,
\end{align}
for all $0\leq a < b \leq T-h$. In fact, by approximation with smooth functions, one sees that  all $\varphi\in L^p(0,T;W^{1,p}_0(\Omega))\cap L^\infty(\Omega_T)$ are admissible in \eqref{trippa}. An analogous identity holds true for the Steklov averages $[\,\cdot\,]_{\bar h}$ and $h \leq a < b \leq T$.

We now prove the existence of a solution to the regularized problem in the sense of Definition \ref{approxdef}. We will follow the functional analytic approach of Showalter \cite{Sho} making use of Galerkin's method. In fact, one can reason as in the proof of \cite[Theorem 4.1, Section III.4]{Sho}, despite the somewhat weaker coercivity condition in our case. For the reader's convenience we present the full argument below. We have opted to avoid the theory of operators of type M used by Showalter, exploiting instead the stronger monotonicity property of the vector field $A_k$.

\begin{lem}\label{lem:approx-sol}
For any $k>1$ there exists at least one admissible weak solution $v_k$ to~\eqref{approxprob} in the sense of Definition~\ref{approxdef}.
\end{lem}

\begin{proof}[Proof]
We fix $k>1$ and consider the modified vector field 
\begin{align*}
\tilde{A}_k(w,\xi):=A_k\big(\tfrac1k + w, \xi\big)=
T_k^\alpha(\tfrac1k + w )|\xi + \nabla z|^{p-2}(\xi + \nabla z),
\end{align*} 
for $w\in\R$ and $\xi\in\R^n$. We prove  the existence of a function $w\in C([0,T];L^2(\Omega))\cap L^p(0,T;W^{1,p}_0(\Omega))$ satisfying
\begin{align}\label{tildeAk_equation}
\iint_{\Omega_T}\big[\tilde{A}_k(w,\nabla w)\cdot \nabla \varphi - w\partial_t\varphi\big]\d x\d t = \iint_{\Omega_T} \big[f\varphi + k^{-\alpha}|\nabla z|^{p-2}\nabla z\cdot \nabla \varphi\big] \d x \d t,
\end{align}
for all $\varphi\in C^\infty_0(\Omega_T)$, and $w(\cdot,0)= \Psi$ in $\Omega$. Then $v_k:= \tfrac1k + w$ is an admissible weak solution in the sense of Definition \ref{approxdef}.  

We denote $V=L^2(\Omega)\cap W^{1,p}_0(\Omega)$ and define $\A : V\to V'$ by
\begin{align*}
\langle \A(w), v\rangle := \int_\Omega \tilde{A}_k(w,\nabla w)\cdot \nabla v \d x,
\hspace{2mm}v,w\in V.
\end{align*}
We define $F\in L^{p'}(0,T;V')$ by setting
\begin{align*}
\langle F(t), v\rangle := \int_\Omega \big[f(\cdot,t) v + k^{-\alpha}|\nabla z|^{p-2}\nabla z\cdot \nabla v\big] \d x, \hspace{3mm} v\in V.
\end{align*}
Recall that we have the inclusions $V \hookrightarrow L^2(\Omega)\hookrightarrow V'$ with $V$ being dense in $L^2(\Omega)$. Then, \eqref{tildeAk_equation} is equivalent to
\begin{align*}
w'+\A(w)=F, \hspace{10mm}\textrm{in } V'.
\end{align*}
This equation can be understood in the weak sense using the Bochner integral, or equivalently pointwise a.e. Pick a basis $(v_j)^{\infty}_{j=1}$ of $V$, and for each $m\in\N$ vectors $\psi_m \in \textrm{span} (v_1,\dots,v_m) =:V_m $ converging to $\Psi$ in $L^2(\Omega)$, and consider for a fixed $m\in \N$ the problem of finding a map $w_m:[0,T]\to  V_m$ satisfying
\begin{align}\label{galerkin}
\left\{
\begin{array}{l}
(w_m'(t),v_j)+\langle\A(w_m(t)),v_j\rangle = \langle F(t),v_j\rangle, \hspace{5mm} \mbox{for $j\in\{1,\dots m\}$ and a.e. $t$.}
\\[5pt]
w_m(0)=\psi_m 
\end{array}
\right.
\end{align}
Here $(\cdot,\cdot)$ denotes the inner product in $L^2(\Omega)$ and $\langle\cdot,\cdot\rangle$ denotes the dual pairing of $V'$ and $V$. We define $g:\R^m\to \R^m$ with components
\begin{align*}
g^i(y)=-\langle\A(\Sigma^m_{j=1}y^j v_j), v_i\rangle,
\quad\mbox{for $y\in\R^m$}
\end{align*}
and $\tilde f:[0,T]\to \R^m$ with components
\begin{align*}
\tilde f^i(t)=\langle F(t),v_i\rangle.
\end{align*}
The dominated convergence theorem shows that $g$ is continuous, $\tilde f$ is evidently integrable and the matrix with components $(v_i,v_j)$, $i,j\in \{1,\dots,m\}$ is invertible. Therefore Lemma~\ref{lem:Cauchy-Peano} guarantees that the problem \eqref{galerkin} has a solution $w_m$ on some interval $[0,\delta]$. 
This solution can be extended to a maximal interval $J\subset [0,T]$. Multiplying \eqref{galerkin} by the component function $w_m^j(t)$ of $w_m$ in the basis $(v_j)^m_{j=1}$ and summing over $j$ we have
\begin{align}\label{summedover}
(w_m'(t),w_m(t))+\langle\A(w_m(t)),w_m(t)\rangle = \langle F(t),w_m(t)\rangle \textrm{ for a.e. } t\in J.
\end{align}
The coercivity property \eqref{coercivity} of $A_k$ implies the same property for $\tilde{A}_k$, from which we obtain for all $v\in V$ that
\begin{align*}
\langle \A(v),v\rangle &\geq 2^{-p}k^{-\alpha}\int_\Omega |\nabla v|^p\d x - 
2^p k^{\alpha}\int_\Omega |\nabla z|^p\d x
\\
&\geq \tfrac{1}{c} \norm{v}^p_{W^{1,p}(\Omega)} - c\|\nabla z\|^p_{L^p(\Omega)},
\end{align*}
with a constant $c> 1$ depending on $\alpha, p, k$ and $\Omega$. Using this estimate in \eqref{summedover} shows that
\begin{align*}
(w_m'(t),w_m(t)) + \tfrac{1}{c}\norm{w_m(t)}^p_{W^{1,p}(\Omega)} \leq c\|\nabla z\|^p_{L^p(\Omega)} + \norm{F(t)}_{W^{-1,p'}(\Omega)}\norm{w_m(t)}_{W^{1,p}(\Omega)},
\end{align*}
for a.e. $t\in J$. Here we have extended $F(t)$ to an element of $W^{-1,p'}(\Omega)=(W^{1,p}_0(\Omega))'$ using the same formula as before. 
Applying Young's inequality to the last term we find that
\begin{align*}
(w_m'(t),w_m(t)) + \tfrac{1}{c}\norm{w_m(t)}^p_{W^{1,p}(\Omega)} \leq c\|\nabla z\|^p_{L^p(\Omega)} + c\norm{F(t)}_{W^{-1,p'}(\Omega)}^{p'},
\end{align*}
where $c>1$ depends on $\alpha, p, k$ and $\Omega$. 
Integrating the last intequality we have
\begin{align}\label{bokachuta}
\tfrac{1}{2}\norm{w_m(t)}_{L^2(\Omega)}^2 & + \tfrac1c\int^t_0 \norm{w_m(s)}^p_{W^{1,p}(\Omega)} \d s \nonumber\\
&\leq c\|\nabla z\|^p_{L^p(\Omega)} T + \tfrac{1}{2}\norm{\psi_m}_{L^2(\Omega)}^2 + c\int^T_0 \norm{F(s)}_{W^{-1,p'}(\Omega)}^{p'} \d s, 
\end{align}
for all $t\in J$. This shows that $\norm{w_m}_{L^2(\Omega)}$ and the component functions $w^i_m$ stay bounded on $J$, and from the system of equations we conclude that $w_m$ is absolutely continous on all of $J$. If $J$ is not the interval $[0,T]$ then $J=[0,b)$ where $b<T$. Then the uniform continuity of $w_m$ and the finite dimension of $V_m$ show that there is a limit of $w_m(t)$ as $t\uparrow b$ which allows us to extend $w_m$, thus contradicting maximality. Hence, $w_m$ must indeed be defined on all of $[0,T]$. Moreover, since $\psi_m\to \Psi$ in $L^2(\Omega)$, the estimate \eqref{bokachuta} shows on the one hand that $(w_m)$ is a bounded sequence in $L^\infty(0,T;L^2(\Omega))$ and on the other hand that $(w_m)$ is a bounded sequence in $L^p(0,T;W^{1,p}(\Omega))$. Hence, we infer that $w_m$ is a bounded sequence in $L^p(0,T;V)$. From the definition of $\A$ and \eqref{Ak-bd} we see that 
\begin{align*}
\norm{\A(v)}_{V'}\leq c\norm{v}^{p-1}_{V}+c\|\nabla z\|^{p-1}_{L^p(\Omega)},
\end{align*}
with $c=c(\alpha,p,k)$, and hence $(\A (w_m))$ is a bounded sequence in $L^{p'}(0,T;V')$. By reflexivity we have a subsequence still labelled as $(w_m)$ which converges weakly to $w\in L^p(0,T;V)$ and for which $(\A (w_m))$ converges weakly to $\xi \in L^{p'}(0,T;V')$. Furthermore, \eqref{bokachuta} shows that $(w_m(T))$ is bounded in $L^2(\Omega)$ so we may assume that $(w_m(T))$ converges weakly to some $w^*\in L^2(\Omega)$. Take now $\varphi \in C^\infty([0,T])$ and $v\in V_m$. From \eqref{galerkin} we see that
\begin{align*}
(w_m'(t),\varphi(t)v)+\langle\A(w_m(t)),\varphi(t)v\rangle &= \langle F(t),\varphi(t)v\rangle.
\end{align*}
Integrating this identity we obtain
\begin{align*}
-\int^T_0 &(w_m(t),v)\varphi'(t)\d t   + \int^T_0 \langle\A(w_m(t)),\varphi(t)v\rangle \d t 
\\
&=  (\psi_m,v)\varphi(0) - (w_m(T),v)\varphi(T) +\int^T_0 \langle F(t),\varphi(t)v\rangle \d t.
\end{align*}
Due to the weak convergences mentioned above we obtain by taking $m\to \infty$ that
\begin{align}\label{kattila}
-\int^T_0 &(w(t),v)\varphi'(t)\d t   + \int^T_0 \langle\xi(t),v\rangle\varphi(t) \d t 
\\
\notag &=  (\Psi,v)\varphi(0)  -   (w^*,v)\varphi(T)+ \int^T_0 \langle F(t),\varphi(t)v\rangle \d t,
\end{align}
for all $v\in V_{m_o}$ for any $m_o\in\N$, and by approximation for all $v\in V$. This shows (by taking $\varphi \in C^\infty_0(0,T)$) that 
\begin{align}\label{eq_for_w}
w' + \xi= F,
\end{align}
in $L^{p'}(0,T;V')$, and thus $w\in C([0,T];L^2(\Omega))$, see also Proposition 1.2 of Section III.1 in \cite{Sho}. Next we show that $w$ satisfies the right initial condition. Using the test function
\begin{align*}
\varphi(t)=\begin{cases}
\tfrac1\varepsilon(\varepsilon-t), &t\in [0,\varepsilon],
\\
0, &t>\varepsilon,
\end{cases}
\end{align*}
in \eqref{kattila} we have for all $v\in V$,
\begin{align*}
\bigg(\frac{1}{\varepsilon}\int^\varepsilon_0 w(t)\d t-\Psi, v \bigg)=\int^\varepsilon_0 \langle F(t) - \xi(t),v\rangle \varphi(t)\d t,
\end{align*}
where the integral on the left-hand side is taken in the Bochner sense of $w$ as an $L^2(\Omega)$-valued map. By the density of $V$ in $L^2(\Omega)$, this implies
\begin{align*}
\bigg\|\frac{1}{\varepsilon}\int^\varepsilon_0 w(t)\d t-\Psi\bigg\|_{L^2(\Omega)}\leq \int^\varepsilon_0 \norm{F(t) - \xi(t)}_{V'} \d t.
\end{align*} 
The right-hand side converges to zero as $\varepsilon\downarrow 0$. On the other hand, since $w\in C([0,T];L^2(\Omega))$, we know that the limit of the integral average appearing on the left-hand side is $w(0)$. Thus we have confirmed that $w(0)=\Psi$. It only remains to show that $\xi = \A(w)$. Since $\xi$ is the weak limit of $(\A(w_m))$, it is sufficient to show that $(\A(w_m))$ converges weakly to $\A(w)$. In order to prove the weak convergence, we first show the  $L^p$-convergence of $(\nabla w_m)$ to $\nabla w$. From the monotonicity condition \eqref{Ak-monotone} satisfied by $A_k$ it follows that
\begin{align*}
\langle \A w_m - \A w, w_m-w\rangle \geq \begin{cases}
c \displaystyle{\iint_{\Omega_T} |\nabla w_m - \nabla w|^p \d x \d t}, &p\geq 2
\\[15pt]
c \displaystyle{\iint_{\Omega_T\cap \{\nabla w \neq \nabla w_m\}}W_m^{p-2}|\nabla w_m-\nabla w|^2\d x \d t}, &p<2,
\end{cases}
\end{align*}
where $c=c(p,k)$ and
\begin{align*}
W_m:=|\nabla w_m + \nabla z|  + |\nabla w+\nabla z|.
\end{align*}
In the case $p<2$, by H\"older's inequality we may estimate 
\begin{align*}
\iint_{\Omega_T} & |\nabla w_m - \nabla w|^p \d x \d t \\
&= \iint_{\Omega_T\cap \{\nabla w \neq \nabla w_m\}}|\nabla w_m - \nabla w|^p W_m^{\frac{p(p-2)}{2}}W_m^{\frac{p(2-p)}{2}}\d x \d t
\\
&\leq \bigg[\iint_{\Omega_T\cap \{\nabla w \neq \nabla w_m\}}W_m^{p-2}|\nabla w_m-\nabla w|^2\d x\d t\bigg]^\frac{p}{2}\bigg[\iint_{\Omega_T}W_m^p\d x\d t\bigg]^\frac{2-p}{2}.
\end{align*}
The last factor is bounded independently of $m$ since $(w_m)$ is bounded in $L^p(0,T;V)$. Thus, setting $\nu=\max\{1,\tfrac2p\}$ we have in any case that
\begin{align*}
&\bigg[\iint_{\Omega_T} |\nabla w_m - \nabla w|^p \d x \d t\bigg]^\nu \\ 
&\quad\leq  c\int_0^T\big\langle \A (w_m) - \A (w), w_m-w\big\rangle\d t
\\ 
&\quad= c\int_0^T\big[\langle \A (w_m), w_m\rangle - \langle \A (w_m), w\rangle -\langle \A (w), w_m - w\rangle\big]\d t
\\
&\quad= c\int_0^T \big[\langle F, w_m\rangle - \langle \A (w_m), w\rangle -\langle \A (w), w_m - w\rangle \big] \d t +
\tfrac12\norm{\psi_m}^2_{L^2(\Omega)} - \tfrac12\norm{w_m(T)}^2_{L^2(\Omega)},
\end{align*}
for a constant $c$ independent of $m$. 
In the last step we have used \eqref{summedover} integrated over $[0,T]$. The weak convergences of $(w_m)$ to $w$ and $(\A (w_m))$ to $\xi$, the norm convergence of $(\psi_m)$, and the weak lower semicontinuity of the norm applied to the term $\norm{w_m(T)}^2_{L^2(\Omega)}$ then show that 
\begin{align*}
\limsup_{m\to \infty} & \bigg[\iint_{\Omega_T} |\nabla w_m - \nabla w|^p \d x \d t\bigg]^\nu \\
&\leq 
c\int_0^T\big[\langle F, w\rangle - \langle \xi, w\rangle\big]\d t  + \tfrac12\norm{\Psi}^2_{L^2(\Omega)} - \tfrac12\norm{w^*}^2_{L^2(\Omega)} = 0,
\end{align*}
where in the last step we have used \eqref{eq_for_w} applied to $w$ and integrated over $[0,T]$. Thus we have obtained the desired $L^p$-convergence of $\nabla w_m$ to $\nabla w$. To see that the weak convergence of $\A (w_m)$ to $\A (w) $ follows from this, we use Lemma \ref{p-laplace-estim} and obtain
\begin{align*}
&|\langle \A (w_m) -\A (w), v\rangle| \\
&\qquad\leq c\iint_{\Omega_T} \big| |\nabla w_m + \nabla z|^{p-2}(\nabla w_m + \nabla z) - |\nabla w + \nabla z|^{p-2}(\nabla w + \nabla z)\big| |\nabla v|\d x \d t
\\
&\qquad\leq c\iint_{\Omega_T}\big(|\nabla w + \nabla z| + |\nabla w_m-\nabla w|\big)^{p-2}|\nabla w_m - \nabla w| |\nabla v| \d x \d t
\\
&\qquad\leq c \iint_{\Omega_T} \big(|\nabla w_m -\nabla w|^{p-1} + b_p|\nabla w + \nabla z|^{p-2}|\nabla w_m -\nabla w|\big)|\nabla v| \d x \d t,
\end{align*}
where $b_p = 0$ if $p<2$ and $b_p=1$ if $p\ge 2$. H\"older's inequality and the $L^p$-convergence of $\nabla w_m$ show that the last expression converges to zero as $m \to \infty$, so we have confirmed that $\A (w) = \xi$.
\end{proof}

\subsection{Properties of the approximating solutions}

In this section we investigate the properties of the approximating solutions obtained in Lemma~\ref{lem:approx-sol}. We prove lower bounds for the solutions $v_k$ in terms of $k$ and upper bounds which are independent of $k$. Moreover, we obtain a uniform bound for the $L^p(\Omega_T)$-norms of the gradients $\nabla v_k^\beta$. 

\begin{lem}\label{lem:lower-bound-vk}
Let $k>1$ and $v_k$ be an admissible weak solution to the Cauchy-Dirichlet problem \eqref{approxprob} in the sense of Definition~\ref{approxdef}. Then we have $v_k\geq \tfrac1k$ a.e.~in $\Omega_T$.
\end{lem}
\begin{proof}[Proof]
We use a comparison principle argument. Consider the version of \eqref{trippa} with Steklov means $[\,\cdot\,]_{\bar h}$, that is
\begin{equation*}
\int^b_a\int_\Omega \Big[\partial_t [v_k]_{\bar h}\varphi + [A_k(v_k,\nabla v_k)]_{\bar h}\cdot \nabla \varphi \Big] \d x \d t = \int^b_a \int_\Omega \big[[f]_{\bar h}\varphi + k^{-\alpha}|\nabla z|^{p-2}\nabla z\cdot\nabla\varphi\big]\d x \d t,
\end{equation*}
for $h<a<b<T$ and the test function $\varphi=H_\delta(\tfrac1k-v_k)$, where
\begin{align*}
H_\delta(s):=\begin{cases}
0, \hspace{5mm}  &s<0
\\
\frac{s}{\delta}, &s\in [0,\delta]
\\
1, & s>\delta.
\end{cases}
\end{align*}
The test function is admissible since $v_k-\tfrac1k\in L^p(0,T; W^{1,p}_0(\Omega))$, and $H_\delta$ and $H_\delta'$ are bounded. We define $G_\delta$ as
\begin{align*}
G_\delta(s):=\int^s_0 H_\delta(\sigma)\d \sigma = \begin{cases}
0, \hspace{5mm} &s<0
\\
\frac{s^2}{2\delta},  &s\in [0,\delta]
\\
s-\frac{\delta}{2}, &s>\delta.
\end{cases}
\end{align*}
By the properties of the Steklov average and the convexity of $G_\delta$, we have 
\begin{align*}
\partial_t\big[G_\delta(\tfrac1k-v_k)\big]_{\bar{h}}(x,t)&=\tfrac1h\big[G_\delta(\tfrac1k-v_k)(x,t)-G_\delta(\tfrac1k-v_k)(x,t-h)\big]
\\
&\leq \tfrac1h H_\delta(\tfrac1k-v_k)(x,t)\big[ (\tfrac1k-v_k)(x,t) - (\tfrac1k-v_k)(x,t-h)\big]
\\
& = -\tfrac1hH_\delta(\tfrac1k-v_k)(x,t)\big[v_k(x,t)-v_k(x,t-h)\big]
\\
& = -\varphi(x,t) \partial_t[v_k]_{\bar{h}}(x,t).
\end{align*} 
Using this estimate and the fact that $f\geq 0$ in the above identity, we obtain
\begin{align*}
\int^b_a\int_\Omega \Big[\partial_t\big[G_\delta(\tfrac1k-v_k)\big]_{\bar{h}}  + \big[k^{-\alpha}|\nabla z|^{p-2}\nabla z - [A_k(v_k,\nabla v_k)]_{\bar{h}}\big]\cdot \nabla \varphi \Big]\d x \d t \leq 0.
\end{align*}
Taking $h\downarrow 0$, we find that
\begin{align}\label{ribollita}
\bigg[\int_\Omega G_\delta(\tfrac1k-v_k) \d x\bigg]^b_a +
\int^b_a\int_\Omega \big[k^{-\alpha}|\nabla z|^{p-2}\nabla z - A_k(v_k,\nabla v_k)\big] \cdot \nabla \varphi \d x \d t  \leq 0
\end{align}
holds true for a.e. $0<a<b<T$. 
Note that
\begin{align*}
\nabla \varphi = -H'_\delta(\tfrac1k-v_k)\nabla v_k=-\delta^{-1}\chi_{\{0<\frac1k-v_k<\delta\}}\nabla v_k= -\delta^{-1}\chi_{\{\frac1k-\delta < v_k < \frac1k\}}\nabla v_k.
\end{align*}
Thus, whenever the second integrand is nonzero we have $v_k<\tfrac1k$ and then 
\begin{align*}
&-\Big[k^{-\alpha}|\nabla z|^{p-2}\nabla z - A_k(v_k,\nabla v_k)\Big] \cdot \nabla v_k
\\
&\qquad = k^{-\alpha}\Big[ |\nabla v_k+\nabla z|^{p-2}(\nabla v_k+\nabla z)-|\nabla z|^{p-2}\nabla z\Big]\cdot (\nabla v_k+\nabla z -\nabla z)\geq 0.
\end{align*}
This shows that the second integral of \eqref{ribollita} is nonnegative so we can drop it. Thus, we end up with
\begin{align*}
\int_\Omega G_\delta(\tfrac1k-v_k) (x,b)\d x\leq \int_\Omega G_\delta(\tfrac1k-v_k) (x,a)\d x 
\end{align*}
for a.e. $0<a<b<T$. Since $v_k\in C([0,T];L^2(\Omega))$ and $v_k(0)=\frac{1}{k}+ \Psi$ we obtain, by taking $a\downarrow 0$ that
\begin{align*}
\int_\Omega G_\delta(\tfrac1k-v_k) (x,b)\d x\leq \int_\Omega G_\delta(-\Psi) (x)\d x = 0,
\end{align*}
where the last equality follows from the fact that $\Psi\geq 0$. Taking the limit $\delta\to 0$, we obtain
\begin{align*}
\int_\Omega (\tfrac1k-v_k)_+(x,b)\d x \leq 0,
\end{align*}
for a.e. $b \in(0,T)$. Thus, $v_k\geq \frac{1}{k}$ a.e. in $\Omega_T$.
\end{proof}
We now intend to show that the approximative solutions $v_k$ are bounded in the $L^\infty$-norm by a constant independent of $k$. In the proof we will make use of the truncation defined in \eqref{truncation} and the function
\begin{align}\label{v-tilde}
\tilde{v}_k:=T_k\circ v_k.
\end{align}
Note that for $s\ge\frac1k$ we have $T_k(s)=\min\{s,k\}$ and hence $\tilde{v}_k=\min\{v_k,k\}$ by Lemma~\ref{lem:lower-bound-vk}.
The proof is divided into three steps. First we establish an energy estimate for $\tilde{v}_k$. We use this result to show that the $L^{\beta p}$-norm of $\tilde{v}_k$ is bounded independently of $k$. Finally, the energy estimate is utilized in a De Giorgi type iteration to obtain a bound in terms of the $L^{\beta p}$-norm of $\tilde{v}_k$, which by the previous observation concludes the proof.
\begin{lem}\label{lem:energy}
Let $k>1$ and $v_k$ be an admissible weak solution to the Cauchy-Dirichlet problem \eqref{approxprob} in the sense of Definition~\ref{approxdef} and let $M\geq \sup_{\bar{\Omega}}\Psi +1$. Then, the function $\tilde{v}_k$ defined in \eqref{v-tilde} satisfies
\begin{align}\label{caccio_tilde_v}
\sup_{\tau\in [0,T]} \int_\Omega & \big( \tilde{v}_k^{\frac{\beta+1}{2}}-M^{\frac{\beta+1}{2}}\big)^2_+(x,\tau)\d x + \iint_{\Omega_T} \big|\nabla (\tilde{v}_k^\beta-M^\beta)_+\big|^p \d x\d t 
\\
\notag &\leq c \iint_{\Omega_T \cap \{\tilde{v}_k>M\}} \big[|\nabla z|^{\beta p}+ f^{p'} + M^{\beta p}\big]\d x\d t,
\end{align}
for a constant $c=c(\beta,p,\Omega)$ which does not depend on $k$.
\end{lem}
\begin{proof}[Proof]
In the case $k\le M$ the inequality trivially holds, since $\tilde v_k\le k\le M$. Therefore, we are left with the case $k> M$. 
For $0<\varepsilon<\tau<\tau+\varepsilon<T$ and $\zeta_{\tau,\varepsilon}$ we define 
\begin{align*}
\zeta_{\tau,\varepsilon}(t)=\begin{cases}
0, \hspace{5mm}  & t<0
\\
\frac1\varepsilon t, & t\in [0,\varepsilon]
\\
1, \hspace{5mm}  &t\in [\varepsilon, \tau]
\\
1-\frac1\varepsilon(t-\tau), &t\in [\tau,\tau+\varepsilon]
\\
0, & t>\tau+\varepsilon.
\end{cases}
\end{align*}
We use the mollified formulation \eqref{trippa} of the differential equation with $a=0$, $b=T-h$ and $\varphi=\zeta_{\tau,\varepsilon}\big(T_k([v_k]_h)^\beta-M^\beta\big)_+$. We take $h$ so small that the factor $\zeta_{\tau,\varepsilon}$ is supported in $[0,T-h]$. Then the boundedness of $T_k([v_k]_h)$ and the chain rule imply that $\varphi \in L^p(0,T;W^{1,p}_0(\Omega))\cap L^\infty(\Omega_T)$, which means that $\varphi$ is admissible as a test function. 

Our goal is to pass to the limit $h\downarrow 0$ in the mollified differential equation, possibly replacing the equality by a suitable estimate. For the elliptic term we conclude that
\begin{align*}
\lim_{h\downarrow 0}\iint_{\Omega_{T-h}}[A_k(v_k,\nabla v_k)]_h\cdot \nabla \varphi \d x \d t = \iint_{\Omega_T} \zeta_{\tau,\varepsilon} A_k(v_k,\nabla v_k)\cdot \nabla \big(\tilde{v}_k^\beta-M^\beta\big)_+ \d x\d t.
\end{align*}
The convergence of $\nabla\varphi$ in the $L^p$-norm can be seen from the chain rule, and the fact that the outer function $s\mapsto (T_k(s)^\beta-M^\beta)_+$ is piecewise $C^1$ with bounded derivative. 
We now introduce the abbreviation
\begin{align*}
g(s):=\big(T_k(s)^\beta-M^\beta\big)_+,\hspace{3mm} G(s):=\int^s_0 g(t)\d t=\begin{cases}
0, \hspace{5mm} &s\leq M
\\[3pt]
\b[s,M], &s\in [M,k]
\\[3pt]
\b[k,M]+(k^\beta-M^\beta)(s-k), &s>k
\end{cases}
\end{align*}
and observe that $\varphi=\zeta_{\tau,\varepsilon}g([v_k]_h)$. Recall that the boundary term $\b$ has been introduced in \eqref{definition:F}. 
This allows us to treat the parabolic term as 
\begin{align*}
\iint_{\Omega_{T-h}} \partial_t [v_k]_h \varphi \d x \d t &= 
\iint_{\Omega_{T-h}} \zeta_{\tau,\varepsilon} \partial_t [v_k]_h g([v_k]_h) \d x \d t = \iint_{\Omega_{T-h}} \zeta_{\tau,\varepsilon} \partial_t G([v_k]_h) \d x \d t 
\\
&= - \iint_{\Omega_{T-h}} \zeta_{\tau,\varepsilon}' G([v_k]_h) \d x \d t 
\\
&\xrightarrow[h\downarrow 0]{}- \iint_{\Omega_T} \zeta_{\tau,\varepsilon}' G(v_k)  \d x \d t 
\\
&= \frac1\varepsilon\int^{\tau+\varepsilon}_\tau \int_\Omega G(v_k)\d x \d t- \frac1\varepsilon\int^{\varepsilon}_0 \int_\Omega G(v_k)\d x \d t,
\\
&\xrightarrow[\varepsilon\downarrow 0]{}\int_\Omega G(v_k)(x,\tau)\d x.
\end{align*}
The limit of the second term vanishes since $G(v_k)(x,0)=G(\frac1k+\Psi(x))=0$ and $M\ge \sup_{\bar{\Omega}}\Psi +1$ by assumption. The limits exist since $\int_\Omega G(v_k)(x,t) \d x$ is continuous with respect to $t$. The continuity can be concluded from the fact that $G$ is Lipschitz and $v_k\in C([0,T]; L^2(\Omega))$. Therefore, we have after passing to the limits $h\downarrow 0$ and $\varepsilon \downarrow 0$ that
\begin{align}\label{caccio-1}
\int_\Omega & G(v_k)(x,\tau)\d x + \iint_{\Omega_\tau}  A_k(v_k,\nabla v_k)\cdot \nabla \big(\tilde{v}_k^\beta-M^\beta\big)_+ \d x\d t \nonumber\\
&= \iint_{\Omega_\tau} \big[f\big(\tilde{v}_k^\beta-M^\beta\big)_+ + k^{-\alpha}|\nabla z|^{p-2}\nabla z\cdot \nabla \big(\tilde{v}_k^\beta-M^\beta\big)_+\big] \d x \d t
\end{align}
holds true for any $\tau\in(0,T]$. 
Using the expression \eqref{A_k-alt-expression} for $A_k$, the fact that $\nabla (\tilde{v}_k^\beta-M^\beta)_+=0$ a.e.~on the sets $\{(x,t)\in\Omega_T:v_k(x,t)\ge k\}$ and $\{(x,t)\in\Omega_T:v_k(x,t)\le M\}$, the chain rule and Lemma~\ref{lem:V}, we have
\begin{align*}
A_k(v_k, \nabla v_k) \cdot \nabla \big(\tilde{v}_k^\beta-M^\beta\big)_+  &= \tilde{v}_k^\alpha |\nabla v_k+\nabla z |^{p-2}(\nabla v_k+\nabla z) \cdot \nabla \big(\tilde{v}_k^\beta-M^\beta\big)_+ 
\\
&= \tilde{v}_k^\alpha |\nabla \tilde{v}_k+\nabla z |^{p-2}(\nabla \tilde{v}_k+\nabla z) \cdot \nabla \big(\tilde{v}_k^\beta-M^\beta\big)_+ 
\\
&= \beta^{1-p} \chi_{\{\tilde{v}_k > M\}}|\nabla \tilde{v}_k^\beta+\beta \tilde{v}_k^{\beta-1}\nabla z|^{p-2}\big(\nabla \tilde{v}_k^\beta+\beta \tilde{v}_k^{\beta-1} \nabla z\big)\cdot \nabla \tilde{v}_k^\beta 
\\
&\geq \beta^{1-p} \chi_{\{\tilde{v}_k > M\}} \big[2^{-p}|\nabla \tilde{v}_k^\beta|^p - 2^p\beta^p \tilde{v}_k^{p(\beta-1)}|\nabla z|^p\big]
\\
&= 2^{-p}\beta^{1-p} \big|\nabla (\tilde{v}_k^\beta-M^\beta)_+\big|^p -2^p\beta \chi_{\{\tilde{v}_k>M\}}\tilde{v}_k^{p(\beta-1)}|\nabla z|^p.
\end{align*}
Moreover, using the definition of $G$ and Lemma \ref{estimates:boundary_terms}\,(i), we can estimate
\begin{align*}
G(v_k)\geq \chi_{\{\tilde{v}_k>M\}}\b[\tilde{v},M]\geq c(\beta)\,\big(\tilde{v}_k^\frac{\beta+1}{2}-M^\frac{\beta+1}{2}\big)^2_+.
\end{align*}
The last term in the integral on the right-hand side of \eqref{caccio-1} can be estimated using the Schwarz inequality, Young's inequality and the fact that $k>1$ as 
\begin{align*}
k^{-\alpha}|\nabla z|^{p-2}\nabla z\cdot \nabla (\tilde{v}_k^\beta-M^\beta)_+ &\leq |\nabla z|^{p-1}|\nabla (\tilde{v}_k^\beta-M^\beta)_+|
\\
&\leq \tfrac{\beta^{1-p}}{2^{p+1}}|\nabla (\tilde{v}_k^\beta-M^\beta)_+|^p +
c(\beta,p)\, \chi_{\{\tilde{v}_k>M\}}|\nabla z|^p .
\end{align*}
The integral over the first term can be included in the first term on the left-hand side of \eqref{caccio-1} and we end up with
\begin{align*}
&\int_\Omega \big(\tilde{v}_k^\frac{\beta+1}{2}-M^\frac{\beta+1}{2}\big)^2_+(x,\tau)\d x + \iint_{\Omega_\tau}  |\nabla (\tilde{v}_k^\beta-M^\beta)_+|^p\d x\d t 
\\
&\qquad\qquad\leq c \iint_{\Omega_\tau\cap \{\tilde{v}_k>M\}} \big[\tilde{v}_k^{p(\beta-1)} |\nabla z|^p +  f(\tilde{v}_k^\beta-M^\beta)\big]  \d x \d t
\end{align*}
for any $\tau\in(0,T]$ and a constant $c=c(\beta,p)$. 
Since $(\tilde{v}_k^\beta-M^\beta)_+\in L^p(0,T;W^{1,p}_0(\Omega))$, we have by Young's and Poincar\'e's inequality for any $\epsilon\in(0,1)$ that
\begin{align*}
&\iint_{\Omega_\tau\cap \{\tilde{v}_k>M\}} f(\tilde{v}_k^\beta-M^\beta)  \d x \d t\\
&\qquad\le
\epsilon\iint_{\Omega_\tau} (\tilde{v}_k^\beta-M^\beta)_+^p \d x \d t + 
\epsilon^{-\frac1{p-1}}\iint_{\Omega_\tau\cap \{\tilde{v}_k>M\}} f^{p'} \d x \d t \\
&\qquad\le
\epsilon c\iint_{\Omega_\tau} |\nabla(\tilde{v}_k^\beta-M^\beta)_+|^p \d x \d t + 
\epsilon^{-\frac1{p-1}}\iint_{\Omega_\tau\cap \{\tilde{v}_k>M\}} f^{p'} \d x \d t
\end{align*}
and 
\begin{align*}
&\iint_{\Omega_\tau\cap \{\tilde{v}_k>M\}} \tilde{v}_k^{p(\beta-1)} |\nabla z|^p \d x \d t \\
&\qquad\le
\epsilon\iint_{\Omega_\tau\cap \{\tilde{v}>M\}} \tilde{v}_k^{\beta p} \d x \d t +
\epsilon^{-\frac1{\beta-1}}\iint_{\Omega_\tau\cap \{\tilde{v}_k>M\}} |\nabla z|^{\beta p} \d x \d t \\
&\qquad\le
\epsilon c\iint_{\Omega_\tau} (\tilde{v}_k^\beta-M^\beta)_+^p  \d x \d t +
c\,\epsilon^{-\frac1{\beta-1}}\iint_{\Omega_\tau\cap \{\tilde{v}_k>M\}} 
\big[|\nabla z|^{\beta p} + M^{\beta p}\big] \d x \d t\\
&\qquad\le
\epsilon c\iint_{\Omega_\tau} |\nabla(\tilde{v}_k^\beta-M^\beta)_+|^p  \d x \d t +
c\,\epsilon^{-\frac1{\beta-1}}\iint_{\Omega_\tau\cap \{\tilde{v}_k>M\}} 
\big[|\nabla z|^{\beta p} + M^{\beta p}\big] \d x \d t.
\end{align*}
Due to the Poincar\'e inequality, the constant $c$ depends on $\Omega$. Choosing $\epsilon$ small enough we can re-absorb the terms involving $|\nabla(\tilde{v}^\beta-M^\beta)_+|^p$ into the left-hand side. This leads us to
\begin{align*}
&\int_\Omega \big(\tilde{v}_k^\frac{\beta+1}{2}-M^\frac{\beta+1}{2}\big)^2_+(x,\tau)\d x + \iint_{\Omega_\tau}  |\nabla (\tilde{v}_k^\beta-M^\beta)_+|^p\d x\d t  
\\
&\qquad\leq c \iint_{\Omega_\tau\cap \{\tilde{v}_k>M\}} \big[|\nabla z|^{\beta p} + f^{p'} + M^{\beta p}\big] \d x \d t.
\end{align*}
In the first term on the right-hand side we take the supremum over $\tau\in [0,T]$, while in the second one we choose $\tau=T$. Proceeding in this way we end up with inequality \eqref{caccio_tilde_v}.
\end{proof}
We utilize the previous lemma to show that the integral of $\tilde{v}_k^{\beta p}$ is bounded independently of $k$. 
\begin{cor}\label{cor:zanzara}
Let $k>1$ and $v_k$ be an admissible weak solution to the Cauchy-Dirichlet problem \eqref{approxprob} in the sense of Definition~\ref{approxdef}. Then, there is a constant $c$ depending only on $\beta$, $p$ and the domain $\Omega$ such that the function $\tilde{v}_k$ defined in \eqref{v-tilde} satisfies
\begin{align*}
\iint_{\Omega_T}\tilde{v}_k^{\beta p}\d x \d t \leq c\,K,
\end{align*}
where 
\begin{align}\label{def:P}
K := \iint_{\Omega_T } \big[|\nabla z|^{\beta p}+ f^{p'}\big] \d x\d t +
\Big(\sup_{\bar{\Omega}}\Psi^{\beta p} + 1\Big) |\Omega_T|.
\end{align}
\end{cor}
\begin{proof}[Proof] 
We define $M_o=\sup_{\bar{\Omega}}\Psi +1$ and observe that $(\tilde{v}_k^\beta-M_o^\beta)_+\in L^p(0,T;W^{1,p}_0(\Omega))$. Combining Poincar\'e's inequality applied slice wise for a.e.~$t\in(0,T)$ and the energy estimate from Lemma~\ref{lem:energy} we obtain
\begin{align*}
\iint_{\Omega_T} (\tilde{v}_k^\beta-M_o^\beta)_+^p\d x\d t &\leq c\iint_{\Omega_T}|\nabla (\tilde{v}_k^\beta-M_o^\beta)_+|^p\d x \d t
\\
&\leq c\iint_{\Omega_T} \big[|\nabla z|^{\beta p} + f^{p'} + M_o^{\beta p}\big]\d x\d t.
\end{align*}
Thus, we have
\begin{align*}
\iint_{\Omega_T} \tilde{v}_k^{\beta p} \d x\d t 
&\leq c \iint_{\Omega_T} \hspace{-1mm}(\tilde{v}_k^\beta-M_o^\beta)_+^p\d x\d t + 
c\,M_o^{\beta p}|\Omega_T|
\\
&\leq c \iint_{\Omega_T} \big[|\nabla z|^{\beta p} + f^{p'}\big]\d x\d t + c \, M_o^{\beta p} | \Omega_T|,
\end{align*}
with a constant $c=c(\beta,p,\Omega)$. 
This is the desired bound for the $L^{\beta p}$-norm of $\tilde v_k$.
\end{proof}

Now we are ready to prove the boundedness result.
\begin{lem}\label{lem:v_k-unifly-bdd}
Let $k>1$ and $v_k$ be an admissible weak solution to the Cauchy-Dirichlet problem \eqref{approxprob} in the sense of Definition~\ref{approxdef}. Then, there is a constant $L>0$ depending only on $n,\beta, p, \Omega_T, f, \psi, z$, and $ \sigma$ (and thus independent of $k$) such that for every $k> L$ we have
\begin{align*}
v_k\leq L
\quad\mbox{a.e.~in $\Omega_T$.}
\end{align*}
\end{lem}
\begin{proof}[Proof]
For $M\geq \sup_{\bar{\Omega}}\Psi+1$ we define the sequences
\begin{align*}
M_j:=M(2-2^{-j})^\frac{2}{\beta+1},\hspace{5mm} Y_j:=\iint_{\Omega_T}\big(\tilde{v}_k^\frac{\beta+1}{2}-M^\frac{\beta+1}{2}_j\big)_+^\frac{2\beta p}{\beta+1}\d x\d t, \hspace{5mm} j\in \N_0,
\end{align*}
where $\tilde v_k$ is defined in \eqref{v-tilde}. 
Furthermore we denote $m:=\frac{\beta+1}{\beta}$ and $A_j:=\Omega_T \cap \{ \tilde{v}_k>M_j\}$. Note that 
\begin{align*}
\nabla \big(\tilde{v}_k^\beta - M_{j+1}^\beta\big)_+ &= \chi_{\{\tilde{v}_k>M_{j+1}\}}\nabla \tilde{v}_k^\beta=\chi_{\{\tilde{v}_k>M_{j+1}\}}\nabla(\tilde{v}_k^\frac{\beta+1}{2})^\frac{2\beta}{\beta+1}
\\
 &=\tfrac{2\beta}{\beta+1}\chi_{\{\tilde{v}_k>M_{j+1}\}}(\tilde{v}_k^\frac{\beta+1}{2})^\frac{\beta-1}{\beta+1}\nabla \tilde{v}_k^\frac{\beta+1}{2}
= \tfrac{2\beta}{\beta+1}\tilde{v}_k^\frac{\beta-1}{2}\nabla \big(\tilde{v}_k^\frac{\beta+1}{2}-M_{j+1}^\frac{\beta+1}{2}\big)_+.
\end{align*} 
Thus,
\begin{align}\label{cleveland}
\big|\nabla \big(\tilde{v}_k^\beta - M_{j+1}^\beta\big)_+\big|&= \tfrac{2\beta}{\beta+1}\tilde{v}_k^\frac{\beta-1}{2}\big|\nabla \big(\tilde{v}_k^\frac{\beta+1}{2}-M_{j+1}^\frac{\beta+1}{2}\big)_+\big|
\\
\notag &\geq \tfrac{2\beta}{\beta+1}\big(\tilde{v}_k^\frac{\beta+1}{2}-M_{j+1}^\frac{\beta+1}{2}\big)_+^{\frac{2\beta}{\beta+1}-1}\big|\nabla \big(\tilde{v}_k^\frac{\beta+1}{2}-M_{j+1}^\frac{\beta+1}{2}\big)_+\big|
\\
\notag &=\big|\nabla \big(\tilde{v}_k^\frac{\beta+1}{2}-M_{j+1}^\frac{\beta+1}{2}\big)_+^\frac{2\beta}{\beta+1}\big|.
\end{align}
Using H\"older's inequality, Gagliardo Nirenberg's inequality from Lemma~\ref{lemma:Gagliardo}, \eqref{cleveland} and the energy estimate from Lemma~\ref{lem:energy} we infer that
\begin{align*}
&Y_{j+1} \leq \bigg[ \iint_{\Omega_T}\Big[\big(\tilde{v}_k^\frac{\beta+1}{2}-M^\frac{\beta+1}{2}_{j+1}\big)_+^\frac{2\beta }{\beta+1} \Big]^{p\frac{n+m}{n}}\d x\d t \bigg]^\frac{n}{n+m}|A_{j+1}|^\frac{m}{n+m}
\\
&\ \leq c \bigg[\sup_{\tau\in [0,T]}\int_\Omega \hspace{-1mm} \big(\tilde{v}_k^\frac{\beta+1}{2}-M^\frac{\beta+1}{2}_{j+1}\big)_+^2(\tau)\d x\bigg]^\frac{p}{n+m} \\
&\ \ \quad\cdot
 \bigg[ \iint_{\Omega_T} \Big|\nabla  \big(\tilde{v}_k^\frac{\beta+1}{2}-M^\frac{\beta+1}{2}_{j+1}\big)_+^\frac{2\beta }{\beta+1}  \Big|^p \d x \d t \bigg]^\frac{n}{n+m}|A_{j+1}|^\frac{m}{n+m}
\\
&\ = c \bigg[\sup_{\tau\in [0,T]}\int_\Omega \big(\tilde{v}_k^\frac{\beta+1}{2}-M^\frac{\beta+1}{2}_{j+1}\big)_+^2(\tau)\d x\bigg]^\frac{p}{n+m} \bigg[ \iint_{\Omega_T} \big|\nabla  \big(\tilde{v}_k^\beta-M^\beta_{j+1}\big)_+\big|^p \d x \d t \bigg]^\frac{n}{n+m}|A_{j+1}|^\frac{m}{n+m}
\\
&\ \leq c\bigg[ \iint_{A_{j+1}} \big[|\nabla z|^{\beta p}+ f^{p'} + M_{j+1}^{\beta p}\big]\d x\d t \bigg]^\frac{n+p}{n+m}|A_{j+1}|^\frac{m}{n+m}.
\end{align*} 
With the abbreviation $G:=|\nabla z|^{\beta p}+ f^{p'}$, we obtain
\begin{align}\label{Y-estim}
Y_{j+1} &\leq c\bigg[\iint_{A_{j+1}} G \d x\d t + M_{j+1}^{\beta p}|A_{j+1}|\bigg]^\frac{n+p}{n+m}|A_{j+1}|^\frac{m}{n+m}
\notag\\
&\leq c\Big[\norm{G}_{L^\sigma(\Omega_T)}|A_{j+1}|^{1-\frac{1}{\sigma}} + M_{j+1}^{\beta p}|A_{j+1}|\Big]^\frac{n+p}{n+m}|A_{j+1}|^\frac{m}{n+m}.
\end{align}
We can estimate the measure of $A_{j+1}$ by noting that
\begin{align*}
|A_{j+1}|&=M^{-\beta p} 2^{(j+1)\frac{2 \beta p}{\beta+1}}|A_{j+1}|\big(M^\frac{\beta+1}{2}_{j+1}-M^\frac{\beta+1}{2}_j\big)^\frac{2\beta p}{\beta + 1}
\\
&\leq M^{-\beta p} 2^{(j+1)\frac{2 \beta p}{\beta+1}}\iint_{A_{j+1}}\big(\tilde{v}_k^\frac{\beta+1}{2}-M^\frac{\beta+1}{2}_j\big)_+^\frac{2\beta p}{\beta + 1}\d x \d t
\\
&\leq M^{-\beta p} 2^{(j+1)\frac{2 \beta p}{\beta+1}}Y_j.
\end{align*}
We use this to estimate the second term in the square brackets of \eqref{Y-estim} and the last factor. The remaining instance of $|A_{j+1}|$ is treated in the same way except that we drop the factor containing $M$, which is possible since $M>1$. In this way we obtain
\begin{align*}
Y_{j+1} &\leq c\Big[\norm{G}_{L^\sigma}2^{j(1-\frac{1}{\sigma})\frac{2 \beta p}{\beta+1}}Y_j^{1-\frac{1}{\sigma}} +  2^{j\frac{2 \beta p}{\beta+1}}Y_j \Big]^\frac{n+p}{n+m}(M^{-\beta p} 2^{j\frac{2 \beta p}{\beta+1}}Y_j)^\frac{m}{n+m}
\\
&\leq c\, M^{-\frac{\beta pm}{n+m}} \Big[\norm{G}_{L^\sigma(\Omega_T)} +  Y_j^\frac{1}{\sigma} \Big]^\frac{n+p}{n+m}  2^{j\frac{2 \beta p}{\beta+1} \frac{n+m+p}{n+m}}Y_j^{(1-\frac{1}{\sigma})\frac{n+p}{n+m}+\frac{m}{n+m}}
\\
& \leq c\, M^{-\frac{\beta pm}{n+m}}\Big[\norm{G}_{L^\sigma(\Omega_T)} +  \norm{\tilde{v}_k}_{L^{\beta p}(\Omega_T)}^\frac{\beta p}{\sigma} \Big]^\frac{n+p}{n+m}  2^{j\frac{2 \beta p}{\beta+1} \frac{n+m+p}{n+m}}Y_j^{1 +\frac{1}{n+m}(p-\frac{1}{\sigma}(n+p))}.
\end{align*}
In view of Corollary \ref{cor:zanzara} this shows that
\begin{align*}
Y_{j+1} 
 \leq c\, M^{-\frac{\beta pm}{n+m}}\Big[\norm{G}_{L^\sigma(\Omega_T)} +  K^\frac{1}{\sigma} \Big]^\frac{n+p}{n+m}  2^{j\frac{2 \beta p}{\beta+1} \frac{n+m+p}{n+m}}Y_j^{1 +\frac{1}{n+m}(p-\frac{1}{\sigma}(n+p))},
\end{align*}
where $K$ is defined in \eqref{def:P}. 
Thus we have verified the iterative estimate of Lemma~\ref{fastconvg} with the choices
\begin{align*}
C:=c\, M^{-\frac{\beta pm}{n+m}} \Big[\norm{G}_{L^\sigma(\Omega_T)} +  K^\frac{1}{\sigma} \Big]^\frac{n+p}{n+m},\hspace{3mm}b:=2^{\frac{2 \beta p}{\beta+1} \frac{n+m+p}{n+m}}   ,\hspace{3mm}\delta:=\tfrac{1}{n+m}(p-\tfrac{1}{\sigma}(n+p)).
\end{align*}
Note that $\delta > 0$ since $\sigma >\frac{n+p}{p}$. In order to apply Lemma~\ref{fastconvg} we also need to have 
\begin{align*}
Y_0\leq C^{-\frac{1}{\delta}}b^{-\frac{1}{\delta^2}}.
\end{align*}
Since $Y_0\leq \norm{\tilde{v}_k}_{L^{\beta p}(\Omega_T)}^{\beta p}\le c\,K$, it is sufficient that 
\begin{align*}
c\, K\leq C^{-\frac{1}{\delta}}b^{-\frac{1}{\delta^2}},
\end{align*}
which, using the definitions of $C$ and $b$ is equivalent to 
\begin{align*}
M\geq \tilde{c}\,K^\frac{\delta(n+m)}{\beta p m} \Big[\norm{G}_{L^\sigma(\Omega_T)} +  K^\frac{1}{\sigma} \Big]^\frac{n+p}{\beta p m} b^{\frac{n+m}{\delta \beta pm}}.
\end{align*}
for a constant $\tilde{c}$ depending on $n, p, \beta, \Omega$. We now choose
\begin{align*}
M:= \max\bigg\{\sup_{\bar{\Omega}}\Psi+1, \tilde{c}\,K^\frac{\delta(n+m)}{\beta p m} \Big[\norm{G}_{L^\sigma(\Omega_T)} +  K^\frac{1}{\sigma} \Big]^\frac{n+p}{\beta p m} b^{\frac{n+m}{\delta \beta pm}} \bigg\}.
\end{align*}
Thus, $M$ is a constant depending only on $n, p, \beta, \Omega_T, f, z, \sigma$. Since $Y_j\to 0$ with this choice we have that 
\begin{align*}
\norm{T_k\circ v_k}_{L^\infty(\Omega_T)}=\norm{\tilde{v}_k}_{L^\infty(\Omega_T)}\leq 2M =: L,
\end{align*}
for all $k$. But this means that for $k>L$ we have
\begin{align*}
\norm{v_k}_{L^\infty(\Omega_T)}\leq L,
\end{align*}
which proves the claim.
\end{proof}

The previous lemma together with Lemma \ref{lem:lower-bound-vk} and the definition of $A_k$ show that 
\begin{align*}
A_k(v_k,\nabla v_k)=A\big(v_k,\nabla v_k^\beta\big),
\end{align*} 
for large $k$. Thus, for large $k$, we deduce from the differential equation \eqref{lampredotto} that $v_k$ satisfies also the equation 
\begin{align}\label{casumarzu}
\iint_{\Omega_T}\big[A\big(v_k,\nabla v_k^\beta\big)\cdot \nabla \varphi - v_k\partial_t\varphi\big] \d x\d t = \iint_{\Omega_T} \big[f\varphi + k^{-\alpha}|\nabla z|^{p-2}\nabla z\cdot \nabla \varphi\big] \d x \d t,
\end{align}
for any $\varphi\in C_0^\infty(\Omega_T)$. We now show a generalization of this result for test functions which do not necessarily vanish at time zero. This result will be used to verify that once we have concluded the existence of a solution to the original equation, it will also satisfy the correct boundary value.
\begin{lem}For sufficiently large $k$ and $\varphi\in C^\infty(\bar{\Omega}\times[0,T])$ with support in $K\times [0,\tau]$ where $K\subset \Omega$ is compact and $\tau\in (0,T)$ we have
\begin{align}\label{v_k_weak_form_with_init_val}
\iint_{\Omega_T}\big[A\big(v_k,\nabla v_k^\beta\big)\cdot \nabla \varphi - v_k\partial_t\varphi\big] \d x\d t &= \iint_{\Omega_T} \big[f\varphi + k^{-\alpha}|\nabla z|^{p-2}\nabla z\cdot \nabla \varphi\big] \d x \d t
\\
\notag &\quad + \int_\Omega v_k(0)\varphi(0)\d x.
\end{align}
\end{lem}
\begin{proof}[Proof]
For $\varphi$ as in the statement of the lemma, we apply \eqref{casumarzu} with the test function $\zeta_\varepsilon \varphi$, where
\begin{align*}
\zeta_\varepsilon(t)=
\begin{cases}
\frac1\varepsilon t, &t\in [0,\varepsilon],
\\
1, &t>\varepsilon,
\end{cases}
\end{align*}
and pass to the limit $\varepsilon\downarrow 0$. This is possible since $v_k \in C([0,T];L^2(\Omega))$.
\end{proof}

Similarly as in Lemma~\ref{lem:weakform_2} we can show that the weak solution $v_k$ also satisfies the following modified weak form.
\begin{lem}\label{lem:weakform_2_mol}
Let $k>1$ and $v_k$ be an admissible weak solution to the Cauchy-Dirichlet problem \eqref{approxprob} in the sense of Definition~\ref{approxdef}. Then, there holds 
\begin{align}\label{weakform_2_k}
	\iint_{\Omega_T} \zeta' \b[v_k,w] \d x\d t 
	&=
	\iint_{\Omega_T} 
	\zeta\Big[\partial_t w^\beta(v_k-w) +
	A(v_k,\nabla v_k^\beta)\cdot \big(\nabla v_k^\beta-\nabla w^\beta\big) \Big] \d x\d t \nonumber\\
	&\quad - 
	\iint_{\Omega_T} \zeta \Big[
	f(v_k^\beta-w^\beta) + k^{-\alpha} |\nabla z|^{p-2}\nabla z \cdot 
	\big(\nabla v_k^\beta-\nabla w^\beta\big)\Big] \d x\d t ,
\end{align}
for all $w^\beta\in k^{-\beta} + L^p(0,T;W_0^{1,p}(\Omega))$ with $\partial_t w^\beta\in L^{\frac{\beta+1}{\beta}}(\Omega_T)$ and any $\zeta\in W^{1,\infty}([0,T],\R_{\ge 0})$ with $\zeta(0)=0=\zeta(T)$.
\end{lem}

The next step is a uniform $L^p$-bound for the gradients of the functions $v_k^\beta$.
\begin{lem}\label{lem:grad_of_vk_unifly_bdd} 
Let $k>1$ and $v_k$ be an admissible weak solution to the Cauchy-Dirichlet problem \eqref{approxprob} in the sense of Definition~\ref{approxdef}. Then, there is a constant $C>0$ depending on $n,\beta, p, \Omega_T, f, \psi$, and $z$ such that 
\begin{align*}
\iint_{\Omega_T}|\nabla v_k^\beta|^p\d x \d t \leq C,
\end{align*}
for all $k\in \N$.
\end{lem}
\begin{proof}[Proof]
For $\delta\in(0,\frac{T}{2})$ we define
\begin{align*}
\zeta_\delta(t):= \begin{cases}\frac1\delta t, \quad &t\in [0,\delta],
\\ 
1, &t\in [\delta,T-\delta],
\\
\frac1\delta(T-t), &t\in [T-\delta,T].
\end{cases}
\end{align*}
and choose $\zeta=\zeta_\delta$ and the comparison function $w=\frac{1}{k}$ in the modified weak form \eqref{weakform_2_k} of the differential equation. Since $\partial_t w=0$ the first term on the right-hand side is zero. Our goal now is to pass to the limit $\delta \downarrow 0$. 
Note that $\nabla w = 0$ and thus
\begin{align*}
\iint_{\Omega_T} \zeta_\delta A(v_k,\nabla v_k^\beta)\cdot \big(\nabla v_k^\beta-\nabla w^\beta\big) \d x \d t 
&\xrightarrow[\delta \downarrow 0]{} \iint_{\Omega_T} A\big(v_k,\nabla v_k^\beta\big)\cdot \nabla v_k^\beta \d x \d t.
\end{align*}
Using the definition of the vector field $A$ in \eqref{def:A}, Lemma~\ref{lem:V} and Lemma \ref{lem:v_k-unifly-bdd} we obtain
\begin{align*}
A\big(v_k,\nabla v_k^\beta\big)\cdot \nabla v_k^\beta 
&\geq 
\beta^{1-p}\big[2^{-p} |\nabla v_k^\beta|^p -2^p \beta^p v_k^{p(\beta-1)}|\nabla z|^p\big]
\\
&\geq 2^{-p}\beta^{1-p} |\nabla v_k^\beta|^p -2^p\beta L^{p(\beta-1)}|\nabla z|^p.
\end{align*}
For the term on the left-hand side we have 
\begin{align*}
	\iint_{\Omega_T} \zeta_\delta' \b[v_k,w] \d x\d t 
	&=
	\frac{1}{\delta} \int_0^\delta\int_{\Omega} \b[v_k,\tfrac{1}{k}] \d x\d t -
	\frac{1}{\delta} \int_{T-\delta}^T\int_{\Omega} \b[v_k,\tfrac{1}{k}] \d x\d t \\
	&\le
	\frac{1}{\delta} \int_0^\delta\int_{\Omega} \b[v_k,\tfrac{1}{k}] \d x\d t \\
	&\xrightarrow[\delta \downarrow 0]{}
	\int_{\Omega} \b[\tfrac1k+\Psi,\tfrac{1}{k}] \d x\d t \\
	&\le
	\frac{1}{\beta+1}\int_\Omega (\Psi+1)^{\beta+1}\d x .
\end{align*}
We were able to omit the integral over $[T-\delta,T]$ since $\b[v,\tfrac{1}{k}]$ is always nonnegative. Passing to the limits in the remaining terms on the right-hand side presents no problems, and taking into account the previous estimates we end up with
\begin{align*}
\iint_{\Omega_T}|\nabla v_k^\beta |^p\d x \d t &\leq c\,T L^{p(\beta-1)}\int_{\Omega}|\nabla z|^p\d x + c \int_\Omega (\Psi+1)^{\beta+1}\d x
\\
&\quad + c\,L^\beta\iint_{\Omega_T} |f| \d x\d t + c\iint_{\Omega_T}|\nabla z|^{p-1} |\nabla v_k^\beta|\d x \d t,
\end{align*}
where $c=c(p,\beta)$. 
Using Young's inequality in the last integral, we finally obtain
\begin{align*}
\iint_{\Omega_T}|\nabla v_k^\beta |^p\d x \d t 
&\leq 
c\bigg[\int_\Omega \big[L^{p(\beta-1)}|\nabla z|^p + (\Psi+1)^{\beta+1}\big]\d x +
L^\beta \iint_{\Omega_T}|f| \d x \d t \bigg],
\end{align*}
for a constant $c$ depending on $\beta$ and $p$. 
\end{proof}

\section{Proof of the main result}\label{sec:existence}

Lemma~\ref{lem:v_k-unifly-bdd} and Lemma~\ref{lem:grad_of_vk_unifly_bdd} show that $(v_k^\beta)$ is a bounded sequence in the reflexive Banach space $L^p(0,T;W^{1,p}(\Omega))$. Therefore there is a subsequence converging weakly to an element of $L^p(0,T;W^{1,p}(\Omega))$. This element can be regarded as a nonnegative function on $\Omega_T$ since every $v_k^\beta$ is nonnegative, and hence we can write the limit as $v^\beta$ for some nonnegative function $v$. (Mazur's lemma provides us with a subsequence of nonnegative functions converging in $L^p(\Omega_T)$ and from this we obtain yet another subsequence converging pointwise a.e. in $\Omega_T$).
We even have a stronger form of convergence.

\begin{lem}\label{lem:strong}
Let $(v_k)$ be the weak solutions to the Cauchy-Dirichlet problems~\eqref{approxprob} in the sense of Definition~\ref{approxdef}. Then, there exists a subsequence $(k_j)_{j\in\N}$ with $k_j\to\infty$ as $j\to\infty$ and a nonnegative function $v\in L^{\infty}(\Omega_T)$ with $v^\beta\in L^p(0,T;W^{1,p}_0(\Omega))$ such that 
\begin{equation*}
	\left\{
	\begin{array}{ll}
	v_{k_j}^\beta\wto v^\beta &\quad
	\mbox{weakly in $L^p(0,T;W^{1,p}(\Omega))$,} \\[5pt]
	v_{k_j}\to v &\quad
	\mbox{strongly in $L^{q}(\Omega_T)$ for any $q\ge 1$ and a.e.~in $\Omega_T$.}
	\end{array}
	\right.
\end{equation*}
\end{lem}
\begin{proof}[Proof]
From Lemma~\ref{lem:grad_of_vk_unifly_bdd} we know that $(v_k^\beta-k^{-\beta})$ is a bounded sequence in the reflexive Banach space $L^p(0,T;W_0^{1,p}(\Omega))$. Therefore there is a subsequence $(k_j)_{j\in\N}$ with $k_j\to\infty$ as $j\to\infty$ and a function $w\in L^p(0,T;W^{1,p}_0(\Omega))$ such that $v_{k_j}^\beta-k^{-\beta}\wto w$ weakly in $L^p(0,T;W^{1,p}(\Omega))$. This implies that also 
\begin{equation}\label{weak-vkj}
	v_{k_j}^\beta\wto w\quad\mbox{weakly in $L^p(0,T;W^{1,p}(\Omega))$.}
\end{equation}
Our next aim is to ensure strong convergence of $(v_{k_j})$ and thereby to identify the limit function as the pointwise a.e.~limit of the subsequence. 
For this purpose we let $\tau\in (0,T)$. For $h\in(0,T-\tau)$ and $\delta\in(0,\min\{\tau,T-\tau-h\})$ we define  
\begin{align*}
\zeta_\delta(t):=\begin{cases}
0, &t<\tau-\delta \\
\frac1\delta(t-\tau+\delta), &t\in [\tau-\delta,\tau] \\
1, &t \in (\tau,\tau+h) \\
\frac1\delta(\tau+h+\delta-t), &t\in [\tau+h,\tau+h+\delta] \\
0, &t> \tau+h+\delta
\end{cases}
\end{align*}
and consider $\varphi \in C^\infty_0(\Omega)$. 
For $k>1$ we use the weak formulation \eqref{casumarzu} with the test function $\zeta_\delta \varphi$ and obtain
\begin{align*}
\frac1\delta\int^\tau_{\tau-\delta} &\int_\Omega v_k \varphi\d x \d t - \frac1\delta\int^{\tau+h+\delta}_{\tau+h}\int_\Omega v_k \varphi \d x \d t 
\\
&= \iint_{\Omega_T} \big[\zeta_\delta A\big(v_k,\nabla v_k^\beta\big)\cdot \nabla \varphi -\zeta_\delta f\varphi -\zeta_\delta k^{-\alpha}|\nabla z|^{p-2}\nabla z \cdot \nabla \varphi \big]\d x \d t.
\end{align*}
Passing to the limit $\delta \to 0$ we see that for any $\tau\in(0,T)$ and $h\in(0,T-\tau)$ there holds
\begin{align*}
\int_\Omega [v_k(\tau)-v_k(\tau+h)]\varphi \d x = \int^{\tau+h}_\tau \hspace{-2mm}\int_\Omega \big[A\big(v_k,\nabla v_k^\beta\big)\cdot \nabla \varphi -f\varphi -k^{-\alpha}|\nabla z|^{p-2}\nabla z \cdot \nabla \varphi\big] \d x \d t.
\end{align*}
Regarding $v_k$ for fixed times as an element of $(W^{1,p}_0(\Omega))'$, and letting $\langle \cdot, \cdot \rangle$ denote the dual pairing of $(W^{1,p}_0(\Omega))'$ and $W^{1,p}_0(\Omega)$ we thus have
\begin{align*}
|\langle v_k(\tau)-v_k(\tau+h),\varphi\rangle| 
&\leq  
\int^{\tau+h}_\tau \int_\Omega \Big[\big[|A(v_k,\nabla v_k^\beta)|+|\nabla z|^{p-1}\big]|\nabla \varphi| + |f||\varphi|\Big]\d x \d t
\\
&\leq  
c\int^{\tau+h}_\tau \int_\Omega \Big[\big[|\nabla v_k|^{p-1}+L^{\alpha}|\nabla z|^{p-1}\big]|\nabla \varphi| + |f||\varphi|\Big]\d x \d t,
\end{align*}
where in the last line we used the bound for $v_k$ from Lemma~\ref{lem:v_k-unifly-bdd}. 
By H\"older's inequality we continue to estimate
\begin{align*}
|\langle v_k(\tau)-v_k(\tau+h), \varphi \rangle| 
&\leq c\int^{\tau+h}_\tau \bigg[\int_\Omega \big[|\nabla v_k^\beta|^p+|\nabla z|^p\big]\d x \bigg]^\frac{p-1}{p}\norm{\nabla \varphi}_{L^p(\Omega)} \d t
\\
&\quad + c\int^{\tau+h}_\tau \bigg[\int_\Omega |f|^\frac{p}{p-1}\d x\bigg]^\frac{p-1}{p}\norm{\varphi}_{L^p(\Omega)}\d t
\\
& \leq c\, h^\frac{1}{p} \norm{\varphi}_{W^{1,p}(\Omega)}\bigg[\iint_{\Omega_T} \big[|\nabla v_k^\beta|^p+|\nabla z|^p + |f|^\frac{p}{p-1}\big]\d x \d t\bigg]^\frac{p-1}{p},
\end{align*}
where $c=c(\alpha,p,L)$. 
In light of Lemma \ref{lem:grad_of_vk_unifly_bdd}, the expression in the square brackets is bounded by a constant independent of $k$. Thus, by the density of $C^\infty_0(\Omega)$ in $W^{1,p}_0(\Omega)$, we have shown that for almost all $\tau\in(0,T)$ and $h\in(0,T-\tau)$ there holds
\begin{align*}
\norm{v_k(\tau)-v_k(\tau+h)}_{(W^{1,p}_0(\Omega))'}\leq c\, h^\frac{1}{p},
\end{align*}
for a constant $c$ independent of $k$.
Moreover, we recall from Lemmas~\ref{lem:v_k-unifly-bdd} and~\ref{lem:grad_of_vk_unifly_bdd} that $v_k$ is a bounded sequence in $L^\infty(\Omega_T)$ and $v_k^\beta$ in $L^p(0,T;W^{1,p}(\Omega))$. 
We have thus ensured that the assumptions of Corollary~\ref{thm:JS-th6m}\,(i) are satisfied with $Y=(W^{1,p}_0(\Omega))'$, $m=\beta$ and $q=\mu=p$.
Therefore, the application of Corollary~\ref{thm:JS-th6m}\,(i) to $(v_{k_j})$ ensures that $(v^\beta_{k_j})$ is relatively compact in $L^{p}(\Omega_T)$. In particular, there exists a strongly convergent subsequence and by virtue of \eqref{weak-vkj} this implies the strong convergence of the whole sequence $(v^\beta_{k_j})$ to the limit function $w$, i.e.\,we have that 
\begin{equation}\label{strong-vkj}
	v_{k_j}^\beta\to w\quad\mbox{strongly in $L^p(\Omega_T)$.}
\end{equation}
By passing to another subsequence we also obtain that $v_k^\beta$ converges to $w$ pointwise a.e.~in $\Omega_T$ and the uniform boundedness of $v_k$ ensures that also $w\in L^\infty(\Omega_T)$. We now define $v\in L^\infty(\Omega_T)$ via $v^\beta=w$, so that $v_k^\beta\wto v^\beta$ weakly in $L^p(0,T;W^{1,p}(\Omega))$ and $v_k^\beta\to v^\beta$ strongly in $L^p(\Omega_T)$ by \eqref{weak-vkj} and \eqref{strong-vkj}. The uniform boundedness of the sequence $v_k$ together with the strong $L^p$-convergence $v_k^\beta\to v^\beta$ imply that $v_k\to v$ strongly in $L^q(\Omega_T)$ for any $q\ge 1$.
\end{proof}

\medskip
The next step of our argument is to show the strong convergence of the gradients.

\begin{lem}\label{nabla-v_k-convg}
Let the assumptions of Lemma~\ref{lem:strong} be in force. Then, for the subsequence $(k_j)_{j\in\N}$ there additionally holds that 
$$
	\nabla v_{k_j}^\beta\to \nabla v^\beta
	\quad \mbox{strongly in $L^p(\Omega\times I)$}
$$
for any closed subinterval $I$ of $(0,T)$.
\end{lem}

\begin{proof}[Proof]
Due to Remark~\ref{rem:monotone} we know that for any $k>1$ there holds
\begin{align*}
[A(v_k,\nabla v_k^\beta)&-A(v_k,\nabla v^\beta)] \cdot (\nabla v_k^\beta - \nabla v^\beta)
\geq 
\begin{cases}
c\,|\nabla v_k^\beta-\nabla v^\beta|^p, \hspace{5mm} &p\geq 2
\\[5pt]
c\, V_k^{p-2}|\nabla v_k^\beta-\nabla v^\beta|^2, & p<2,
\end{cases}
\end{align*}
for a constant $c=c(p)$ and where
\begin{align*}
V_k:=\big(|\nabla v_k^\beta+\beta v_k^{\beta-1}\nabla z|^2 + |\nabla v^\beta+\beta v_k^{\beta-1}\nabla z|^2\big)^\frac12.
\end{align*}
The expression $V_k^{p-2}$ which appears in the case $p<2$ is not defined if $V_k=0$, but this can occur only if $\nabla v^\beta=\nabla v_k^\beta$, and will therefore cause no problems. 
Now, let $\zeta \in C^\infty_0((0,T);[0,1])$. When $p<2$, H\"older's inequality implies that 
\begin{align*}
\iint_{\Omega_T} & \zeta|\nabla v_k^\beta-\nabla v^\beta|^p\d x \d t \\
&\leq \iint_{\Omega_T \cap \{ \nabla v_k^\beta \neq \nabla v^\beta \} } \zeta|\nabla v_k^\beta-\nabla v^\beta|^p  V_k^{\frac{p}{2}(p-2)} V_k^{\frac{p}{2}(2-p)}\d x \d t
\\
&\leq \bigg[\iint_{\Omega_T \cap \{ \nabla v_k^\beta \neq \nabla v^\beta \} } \zeta V_k^{p-2}|\nabla v_k^\beta-\nabla v^\beta|^2\d x \d t\bigg]^\frac{p}{2}\bigg[\iint_{\Omega_T}\zeta V_k^p\d x \d t\bigg]^\frac{2-p}{2}
\\
&\leq c\bigg[\iint_{\Omega_T} \zeta[A(v_k,\nabla v_k^\beta)-A(v_k,\nabla v^\beta)] \cdot (\nabla v_k^\beta - \nabla v^\beta)\d x \d t\bigg]^\frac{p}{2}
\\
&\quad \cdot\bigg[\iint_{\Omega_T}|\nabla v_k^\beta|^p + |\nabla v^\beta|^p + v_k^{p(\beta-1)}|\nabla z|^p\d x \d t  \bigg]^\frac{2-p}{2}.
\end{align*}
By Lemma \ref{lem:v_k-unifly-bdd} and Lemma \ref{lem:grad_of_vk_unifly_bdd}, the expression on the last row is bounded by a constant independent of $k$. The previous observations show that regardless of the value of $p$,
\begin{align}\label{kiwi}
\notag \frac{1}{c}\bigg[  \iint_{\Omega_T} &\zeta|\nabla v_k^\beta  -\nabla v^\beta|^p\d x \d t \bigg]^\nu \\\notag
&\leq \iint_{\Omega_T} \zeta[A(v_k,\nabla v_k^\beta)-A(v_k,\nabla v^\beta)] \cdot (\nabla v_k^\beta - \nabla v^\beta) \d x \d t 
\\
&  = \iint_{\Omega_T} \zeta A(v_k,\nabla v_k^\beta) \cdot (\nabla v_k^\beta - \nabla v^\beta) \d x \d t - \iint_{\Omega_T} \zeta A(v,\nabla v^\beta) \cdot (\nabla v_k^\beta - \nabla v^\beta)\d x \d t
\notag\\ 
&\quad + \iint_{\Omega_T} \zeta[A(v,\nabla v^\beta)-A(v_k,\nabla v^\beta)] \cdot (\nabla v_k^\beta - \nabla v^\beta) \d x \d t \notag\\
&=:
\mbox{I}_k + \mbox{II}_k + \mbox{III}_k,
\end{align}
where $\nu= \max\{1,\frac{2}{p}\}$ and the constant $c$ is independent of $k$.
From Lemma~\ref{lem:strong} we know that the subsequence $(v_{k_j}^\beta)$ converges weakly in $L^p(0,T;W^{1,p}(\Omega))$ to $v^\beta$ and that $v_{k_j}$ converges strongly in $L^{q}(\Omega_T)$ for any $q\ge 1$ and pointwise a.e.~in $\Omega_T$ to $v$. 
Since $\zeta A(v,\nabla v^\beta)\in L^{p'}(\Omega_T;\R^n)$, the weak convergence $v_{k_j}^\beta\wto v^\beta$ in $L^p(0,T;W^{1,p}(\Omega))$ shows that the second term on the right-hand side of \eqref{kiwi} vanishes in the limit $j\to\infty$. More precisely, we have
\begin{align*}
\lim_{j\to\infty} \mbox{II}_{k_j}
\equiv
\lim_{j\to\infty} \iint_{\Omega_T} \zeta A(v,\nabla v^\beta) \cdot (\nabla v_{k_j}^\beta - \nabla v^\beta)\d x \d t 
=
0.
\end{align*}
To treat the third term on the right-hand side of \eqref{kiwi}, we note that Lemma \ref{lem:v_k-unifly-bdd}, the pointwise a.e.~convergence $v_{k_j}\to v$ and the dominated convergence theorem guarantee that $A(v_{k_j},\nabla v^\beta)$ converges to $A(v,\nabla v^\beta)$ strongly in $L^{p'}$. The quantity $\nabla v_k^\beta-\nabla v^\beta$ stays bounded in $L^p$ due to Lemma \ref{lem:grad_of_vk_unifly_bdd}, so by H\"older's inequality we conclude that 
\begin{align*}
\lim_{j\to\infty} \mbox{III}_{k_j}
\equiv
\lim_{j\to\infty} \iint_{\Omega_T} \zeta\big[A(v,\nabla v^\beta)-A(v_{k_j},\nabla v^\beta)\big] \cdot (\nabla v_{k_j}^\beta - \nabla v^\beta) \d x \d t
=
0.
\end{align*}
It only remains to control the first term on the right-hand side of \eqref{kiwi}. This term can be re-written as 
\begin{align*}
\mbox{I}_{k} 
&= \iint_{\Omega_T} \zeta A(v_k,\nabla v_k^\beta) \cdot \big(\nabla v_k^\beta - \nabla \Ex{v^\beta}\big) \d x \d t +  \iint_{\Omega_T} \zeta A(v_k,\nabla v_k^\beta) \cdot \big(\nabla \Ex{v^\beta}-\nabla v^\beta\big) \d x \d t \\
&=: 
\mbox{I}_{k}^{(1)} + \mbox{I}_{k}^{(2)},
\end{align*}
where $\Ex{v^\beta}$ denotes the exponential time-mollification defined in \eqref{def:moll}.
Since $A(v_k,\nabla v_k^\beta)$ is bounded in $L^{p'}$ uniformly in $k$ we can use H\"older's inequality to estimate the second term on the right-hand side as
\begin{align*}
|\mbox{I}_{k}^{(2)}|
\le
\| A(v_k,\nabla v_k^\beta)\|_{L^{p'}(\Omega_T)} \norm{\nabla \Ex{v^\beta}-\nabla v^\beta}_{L^p(\Omega_T)} 
\leq 
c\,\norm{\nabla \Ex{v^\beta}-\nabla v^\beta}_{L^p(\Omega_T)},
\end{align*}
with a constant $c$ independent of $k$. 
In order to treat the first term, we use the modified weak formulation \eqref{weakform_2_k} with the comparision function $w_{h,k}$ defined by
$w_{h,k}^\beta=k^{-\beta}+\Ex{v^\beta}$.
The choice of comparison function requires some justification. We know that $k^{-\beta}+\Ex{v^\beta}\in k^{-\beta}+L^p(0,T;W_0^{1,p}(\Omega))$ since $v^\beta\in L^p(0,T;W_0^{1,p}(\Omega))$. By Lemma~\ref{expmolproperties}\,(iii) the exponential time mollification preserves this space. Moreover, by Lemma ~\ref{expmolproperties}\,(ii) and the boundedness of $v$ we see that $\partial_t w_{h,k}^\beta \in L^\frac{\beta+1}{\beta}(\Omega_T)$. Thus, the comparison function $w_{h,k}$ is admissible, so that
\begin{align*}
	\mbox{I}_{k}^{(1)}
	&=
	\iint_{\Omega_T} \Big[\zeta' \b\big[v_k,w_{h,k}\big] -
	\zeta \partial_t \Ex{v^\beta}(v_k-w_{h,k}) \Big]\d x\d t \\
	&\quad +
	\iint_{\Omega_T} \zeta \Big[
	f\big(v_k^\beta-w_{h,k}^\beta\big) + k^{-\alpha} |\nabla z|^{p-2}\nabla z \cdot 
	\big(\nabla v_k^\beta-\nabla \Ex{v^\beta}\big)\Big] \d x\d t .
\end{align*}
By the convergence properties of $v_{k_j}$ we obtain
\begin{align*}
 \limsup_{j\to\infty} \mbox{I}_{k_j}^{(1)}
& = \iint_{\Omega_T} \Big[\zeta' \b\big[v,(\Ex{v^\beta})^{\frac{1}{\beta}}\big] -
	\zeta \partial_t \Ex{v^\beta}\big(v-(\Ex{v^\beta})^{\frac{1}{\beta}}\big) \Big]\d x\d t \\
&\quad+ 
\iint_{\Omega_T} \zeta f \big(v^\beta -\Ex{v^\beta}\big) \d x \d t \\
&\le 
\iint_{\Omega_T} \zeta' \b\big[v,(\Ex{v^\beta})^{\frac{1}{\beta}}\big] \d x\d t + 
\iint_{\Omega_T} \zeta f \big(v^\beta -\Ex{v^\beta}\big) \d x \d t,
\end{align*}
where in the last line we used Lemma \ref{expmolproperties}\,(ii).
Combining the last two estimates and joining the previously obtained bounds and convergence properties for $\mbox{I}_{k_j}$ -- $\mbox{III}_{k_j}$ with \eqref{kiwi}, we end up with
\begin{align*}
&\limsup_{j\to\infty} \bigg[\iint_{\Omega_T} \zeta|\nabla v_{k_j}^\beta  -\nabla v^\beta|^p\d x \d t\bigg]^\nu \\
&\quad\leq  c\big\|\nabla \Ex{v^\beta}-\nabla v^\beta\big\|_{L^p(\Omega_T)}+
c\iint_{\Omega_T} \Big[\zeta' \b\big[v,(\Ex{v^\beta})^{\frac{1}{\beta}}\big]  +  \zeta f \big(v^\beta -\Ex{v^\beta}\big)\Big] \d x \d t.
\end{align*}
By Lemma \ref{expmolproperties}\,(iii) we see that the right-hand side converges to zero as $h\to 0$. Thus, we have shown that
\begin{align*}
\lim_{j\to \infty}\iint_{\Omega_T} \zeta|\nabla v_{k_j}^\beta  -\nabla v^\beta|^p\d x \d t = 0.
\end{align*}
This proves the claim of the lemma, since for any closed subinterval $I \subset (0,T)$ we can choose $\zeta \in C^\infty_0(0,T;[0,1])$ such that $\xi|_I=1$. In this case $\chi_I\leq \xi$ and the result follows.
\end{proof}

\begin{lem}\label{lem:huanggua}
The function $v$ obtained in Lemma~\ref{lem:strong} satisfies
\begin{align}\label{huanggua}
\iint_{\Omega_T} \big[A(v,\nabla v^\beta)\cdot \nabla \varphi - v\partial_t \varphi\big]\d x\d t =\iint_{\Omega_T} f\varphi \d x\d t + \int_\Omega \Psi\varphi(0)\d x,
\end{align}
for every $\varphi \in C^\infty(\bar{\Omega}\times[0,T])$ with support contained in $K\times [0,\tau]$ where $K\subset \Omega$ is compact and $\tau\in(0,T)$.
\end{lem}

\begin{proof}[Proof]
We fix a test function $\varphi$ as above, and note that it is sufficient to consider functions satisfying $|\nabla \varphi|\leq 1$. Recall that $v_{k_j}$ satisfies \eqref{v_k_weak_form_with_init_val} with $k=k_j$ and $j$ sufficiently large. The goal is to pass to the limit $j\to \infty$. The limit
\begin{align*}
\lim_{j\to \infty} \iint_{\Omega_T} v_{k_j}\partial_t \varphi\d x \d t = \iint_{\Omega_T} v\partial_t \varphi\d x \d t
\end{align*}
follows from the $L^1$-convergence of $v_{k_j}$ to $v$ and the limits 
\begin{align*}
\lim_{j\to\infty} k_j^{-\alpha}\iint_{\Omega_T} |\nabla z|^{p-2}\nabla z\cdot \nabla \varphi \d x \d t = 0, \hspace{8mm} \lim_{j\to\infty}\int_\Omega v_{k_j}(0)\varphi(0)\d x = \int_\Omega \Psi\varphi(0)\d x,
\end{align*}
are trivial (recall that $v_k(0)=\Psi+\tfrac1k$). It remains to treat the elliptic term. For this, we abbreviate 
$$
	\mathcal A_k
	:=
	A(v_k,\nabla v_k^\beta) - A(v,\nabla v^\beta)
$$
and note that for any $\delta \in (0,\tau)$ we can estimate 
\begin{align}\label{magyar}
 \iint_{\Omega_T} &|A(v_k,\nabla v_k^\beta)\cdot \nabla \varphi - A(v,\nabla v^\beta)\cdot \nabla \varphi|\d x \d t 
\\
\notag &\leq  \iint_{\Omega\times [0,\delta]} |\mathcal A_k|\d x \d t
+ \iint_{\Omega\times [\delta,\tau]} |\mathcal A_k|\d x \d t,
\end{align}
where we have used the Cauchy-Schwarz inequality and the bound on $\nabla \varphi$. Let $\varepsilon>0$. Due to the definition of $A$, and the uniform bounds obtained in Lemma \ref{lem:v_k-unifly-bdd} and Lemma \ref{lem:grad_of_vk_unifly_bdd}, the integrand in the first term on the right-hand side of \eqref{magyar} is bounded in the $L^{p'}$-norm independently of $k$, so H\"older's inequality allows us to conclude that for sufficiently small $\delta>0$, the bound
\begin{align}\label{mongolia}
\iint_{\Omega\times [0,\delta]} |\mathcal A_k|\d x \d t<\frac{\varepsilon}{2},
\end{align}
is satisfied independently of $k\in \N$. With such a fixed $\delta$ we now estimate the second term on the right-hand side of \eqref{magyar}  as 
\begin{align}\label{terms}
\notag \iint_{\Omega\times [\delta,\tau]} &|\mathcal A_k|\d x \d t
\\
&\hspace{-5mm}\leq \iint_{\Omega\times [\delta,\tau]} |A(v_k,\nabla v_k^\beta) - A(v_k,\nabla v^\beta)|\d x \d t + \iint_{\Omega\times [\delta,\tau]} |A(v_k,\nabla v^\beta) - A(v,\nabla v^\beta)|\d x \d t.
\end{align}
We will use Lemma \ref{p-laplace-estim} to treat both terms. This lemma shows that the integrand in the first term may be estimated as
\begin{align*}
|A(v_k,\nabla v_k^\beta) - A(v_k,\nabla v^\beta)|&\leq c \big(|\nabla v^\beta+ \beta v_k^{\beta-1}\nabla z| + |\nabla v^\beta-\nabla v_k^\beta|\big)^{p-2}|\nabla v^\beta-\nabla v_k^\beta|
\\
&\leq  c|\nabla v^\beta-\nabla v_k^\beta|^{p-1}+c\,b_p|\nabla v^\beta+ \beta v_k^{\beta-1}\nabla z|^{p-2}|\nabla v^\beta-\nabla v_k^\beta|,
\end{align*}
where $c=c(p)$ and $b_p=0$ if $p<2$ and $b_p=1$ if $p\ge 2$. Convergence of the integral of the first term is clear as $k=k_j\to \infty$ due to Lemma \ref{nabla-v_k-convg}. In the case $p\geq 2$ the integral of the second term is treated using Lemma \ref{lem:v_k-unifly-bdd}, Lemma \ref{nabla-v_k-convg} and H\"older's inequality:
\begin{align*}
\iint_{\Omega\times [\delta,\tau]} & |\nabla v^\beta+ \beta v_k^{\beta-1}\nabla z|^{p-2}|\nabla v^\beta-\nabla v_k^\beta| \d x\d t 
\\
&\leq \bigg[\iint_{\Omega\times [\delta,\tau]}(|\nabla v^\beta|+\beta L^{\beta-1}|\nabla z |)^\frac{p(p-2)}{p-1} \d x \d t\bigg]^\frac{p-1}{p} \norm{\nabla v^\beta-\nabla v_k^\beta}_{L^p(\Omega\times [\delta,\tau])}.
\end{align*}
The exponent of the integral is less than $p$, so the integral is a finite number. These considerations show that the first term of \eqref{terms} converges to zero as $k\in \{k_j\,|\,j\in \N\}$ approaches infinity. The integrand in the second term of \eqref{terms} can be estimated using Lemma \ref{p-laplace-estim} as
\begin{align*}
|A(v_k,\nabla v^\beta) - A(v,\nabla v^\beta)| &\leq c\big(|\nabla v^\beta+\beta v^{\beta-1}\nabla z| + |\nabla z||v^{\beta-1}-v_k^{\beta-1}|\big)^{p-2}|\nabla z||v^{\beta-1}-v_k^{\beta-1}|
\\
&\leq c |\nabla z|^{p-1} |v^{\beta-1}-v_k^{\beta-1}|^{p-1} \\
&\quad+
c\,b_p |\nabla v^\beta+\beta v^{\beta-1}\nabla z|^{p-2}|\nabla z||v^{\beta-1}-v_k^{\beta-1}|,
\end{align*}
where $c=c(p)$. The terms on the right-hand side can be treated by the dominated convergence theorem. Hence, for $k=k_j$ and for all sufficiently large $j$, we have 
\begin{align}\label{punapapu}
\iint_{\Omega\times [\delta,\tau]} |\mathcal A_k|\d x \d t < \frac{\varepsilon}{2}.
\end{align}
Taking into account \eqref{magyar}, \eqref{mongolia} and \eqref{punapapu}, we have shown that
\begin{align*}
\lim_{j\to \infty} \iint_{\Omega_T}|A(v_{k_j},\nabla v_{k_j}^\beta)\cdot \nabla \varphi - A(v,\nabla v^\beta)\cdot \nabla \varphi|\d x \d t = 0.
\end{align*}
Taking into account all these limits, we have confirmed \eqref{huanggua}. 
\end{proof}

\medskip
Now, we are ready to prove the main theorem.

\begin{proof}[Proof of Theorem~\ref{thm:existence}]
We will show that the function $v$ obtained in Lemma~\ref{lem:strong} is a weak solution to the Cauchy-Dirichlet problem~\eqref{CD-v} in the sense of Definition~\ref{weakdef}. 
Lemma~\ref{lem:huanggua} shows that \eqref{weakform} is valid, since for test functions $\varphi \in C^\infty_0(\Omega_T)$ the second term on the right-hand side of \eqref{huanggua} vanishes. Moreover, by Lemma \ref{lem:time-cont}, we know that $v\in C([0,T];L^{\beta+1}(\Omega))$. It only remains to show that $v(0)=\Psi$. Using \eqref{huanggua} with the test function $\varphi(x)\zeta_\varepsilon(t)$ where $\varphi\in C^\infty_0(\Omega)$ and
\begin{align*}
\zeta_\varepsilon(t)=\begin{cases}
\frac1\varepsilon(\varepsilon-t), &t\in [0,\varepsilon],
\\
0, &t\geq \varepsilon,
\end{cases}
\end{align*}
one obtains
\begin{align*}
\lim_{\varepsilon\to 0} \frac1\varepsilon\int^\varepsilon	_0\int_\Omega v(x,t)\varphi(x)\d x\d t  = \int_\Omega \Psi\varphi\d x.
\end{align*}
On the other hand, since $v\in C([0,T];L^{\beta+1}(\Omega))$, we see that also
\begin{align*}
\lim_{\varepsilon\to 0} \frac1\varepsilon\int^\varepsilon	_0\int_\Omega v(x,t)\varphi(x)\d x\d t  = \int_\Omega v(0)\varphi\d x
\end{align*}
holds. Since $\varphi\in C^\infty_0(\Omega)$ is arbitrary, it follows that $v(0)=\Psi$. 
\end{proof}

\end{document}